\documentclass[a4paper,11pt,DIV=11,% decrease borders by 2 (std 10 for 11pt)
abstract=on% return to old style centered abstract
]{scrartcl}
\usepackage{algorithm, booktabs}
\usepackage{mathtools}
\mathtoolsset{showonlyrefs}
\usepackage{amsfonts}
\usepackage{amsmath}
\usepackage{amssymb}
\usepackage{faktor}
\usepackage{amscd}
\usepackage{array}
\usepackage{amsthm}
\usepackage{subfigure}
\usepackage{mathdots}
\usepackage{mathrsfs}
\usepackage{algorithm}
\usepackage[noend]{algpseudocode}
\usepackage{graphicx}
\usepackage{color}
\usepackage{pgf}
\usepackage{multirow}
\usepackage{listings}
\usepackage{tikz}
\usepackage{graphicx}
\usepackage[left=2.5cm, right=2.5cm,top=2.5cm,bottom=3.5cm]{geometry}
\usepackage{scalerel}

\newmuskip\pFqmuskip

\newcommand*\pFq[6][8]{
  \begingroup % only local assignments
  \pFqmuskip=#1mu\relax
  \begingroup\lccode`\~=`\,
  \lowercase{\endgroup\let~}\pFqcomma
  {}_{#2}F_{#3}{\left(\genfrac..{0pt}{}{#4}{#5};#6\right)}%
  \endgroup
}
\newcommand*\pRegFq[6][8]{
  \begingroup % only local assignments
  \pFqmuskip=#1mu\relax
  \begingroup\lccode`\~=`\,
  \lowercase{\endgroup\let~}\pFqcomma
  {}_{#2}\tilde{F}_{#3}{\left(\genfrac..{0pt}{}{#4}{#5};#6\right)}%
  \endgroup
}

\newcommand{\pFqcomma}{\mskip\pFqmuskip}

\newtheorem{thm}{Theorem}[section]

\newtheorem{rem}[thm]{Remark}

\newtheorem{lemma}[thm]{Lemma}
 
\newtheorem{prop}[thm]{Proposition}

\def\tT{{\mbox{\tiny{T}}}}

\def\cf{{\mathcal F}}

\def\SO{{\mathrm{SO}}}
\def\Oo{{\mathrm{O}}}

\def\C{{\mathbb C}}
\def\N{{\mathbb N}}
\def\R{{\mathbb R}}

\DeclareMathOperator{\Span}{span}
\DeclareMathOperator{\dom}{dom}
\DeclareMathOperator*{\esssup}{ess\,sup}

\newcommand{\red}[1]{#1}
\newcommand{\blue}[1]{#1}

\begin{document}
\title{Continuous Wavelet Frames on the Sphere:\\
The Group-Theoretic Approach Revisited
}
\author{S. Dahlke\footnotemark[1]
\and
F. De Mari \footnotemark[2]
	\and
	E. De Vito\footnotemark[2]
	\and
	M. Hansen\footnotemark[1]	
	\and
	M. Hasannasab\footnotemark[3]	
	\and
	M. Quellmalz
	\footnotemark[3]
	\and 
	G. Steidl\footnotemark[3]
	\and 
	G. Teschke
	\footnotemark[4]
}

\maketitle

\footnotetext[1]{FB12 Mathematik und Informatik, Philipps-Universit\"{a}t Marburg, Hans-Meerwein Strasse,
	Lahnberge, 35032 Marburg, Germany,
	\{dahlke, hansen\}@mathematik.uni-marburg.de.}	
	\footnotetext[2]{Dipartimento di Matematica and MaLGa center, Universit\`a di Genova, Via Dodecaneso 35, Genova, Italy, \{demari, devito\}@dima.unige.it.}
	
	\footnotetext[3]{ Institute of Mathematics, TU Berlin, Straße des 17. Juni 136, 10623 Berlin, Germany,
	\{hasannas, steidl, quellmalz\}@math.tu-berlin.de.}
	\footnotetext[4]{Hochschule Neubrandenburg
University of Applied Sciences, Brodaer Straße 2,
17033 Neubrandenburg, Germany, teschke@hs-nb.de.}

\begin{abstract}
\red{ In \cite{AV99}, Antoine and  Vandergheynst propose a
group-theoretic approach to continuous wavelet frames on the
sphere.   The frame is constructed from a single so-called admissible function 
by applying the unitary operators associated to a representation of the
Lorentz group, which is square-integrable modulo the nilpotent factor
of the Iwasawa decomposition. 
We prove necessary and sufficient conditions for functions on the
sphere, which ensure that the corresponding system is a frame.   We strengthen a similar result in \cite{AV99} by
providing a complete and detailed proof. }
\end{abstract}

%----------------------------------------------------------------------------
\section{Introduction} \label{sec:intro}
%----------------------------------------------------------------------------
In the recent twenty years, continuous and discrete  wavelet frames on the sphere were examined
with different approaches and in particular with various techniques for imitating a dilation on the sphere.
The methods range from extending  the discrete wavelet scheme based on multiresolution to spheres \cite{DDSW1995,Go1999}, to 
Fourier analytic ones with appropriately weighted sums of spherical harmonics 
\cite{FSG1997,PST1996,PT1995,MNPW2000}, lifted spherical wavelets \cite{Sweldens1996},
and constructions on tangent bundles of the sphere 
\cite{DM1996}. 
Further, a spherical wavelet transform based on an integral transform 
with a singular kernel was proposed
in \cite{holschneider1996continuous}. 
For more recent papers following the approximate identities idea, resp. using singular integrals,
we refer to \cite{Bernstein2009,BE2010,Nowak2015}. 
In connection with coorbit spaces spherical wavelets were considered in
\cite{DST2004}.
Various applications of spherical wavelet frames can be found, 
e.g., in \cite{FW1997,NW1996,SBHBKP2005}.

\red{In \cite{AV99} Antoine and  Vandergheynst  \blue{proposed} a
  group-theoretic approach} that generalizes  that of
Grossmann et al. \cite{GMP1985,GMP1986} from locally compact groups to the homogeneous space 
given by a quotient of the Lorentz group. 
For the 2-sphere, the construction and analysis of continuous wavelet frames was outlined in \cite{AV99} 
and for general $n$-spheres in~\cite{antoine1998wavelets}.
\red{In this paper, we exclusively focus on this construction and  
our aim} consists in \blue{adding missing information
to the analysis in \cite{AV99}, which in our opinion appears to be important and non-trivial}.
We will clearly indicate the differences and additions we made by corresponding remarks.
In particular, it appears that finding an admissible function such that the
corresponding function system \blue{indeed has upper} and lower frame \red{bounds}
is quite involved. We consider these results as the main contribution of our paper.

The outline of this paper is as follows. 
In Section \ref{sec:preliminaries}, 
we set the stage by giving the formal group-theoretic definition of wavelets on the sphere. 
Our main result is stated in Section \ref{sec:main}.
Section \ref{sec:nec-suff-condition-for-admissibilty} 
contains a proof of general necessary and sufficient conditions for a
function to be admissible. 
These conditions rely on a sequence of numbers
that has to be bounded from above in order to get an upper frame bound and 
uniformly bounded from below away from zero to have a lower frame bound.
While it is not hard to deduce these general sequence conditions, and indeed they were already provided in \cite{AV99}, 
the main work consists in finding functions which fulfill these conditions.
This is the content of the subsequent sections.
In Section \ref{asymptotics-for-zonal}, we derive necessary and sufficient conditions on functions such that an
upper frame bound for the corresponding continuous frame can be ensured.
To this end, we explore the asymptotic behavior of zonal projections.
In Section \ref{ssec:stereo}, we adopt the nice intuition of Antoine and Vandergheynst to
switch to stereographic projections. 
We show that the derived conditions on the admissible function can be rewritten 
as moment conditions on the isometrically transferred
admissible function to the plane.
Finally, the lower frame bound is treated in Section \ref{sec:lower-frame-bound}. 
As known from frame theory, conditions for lower frame bounds are typically 
more complicated than the Bessel condition, which is also the case for our setting.

%--------------------------------------------------------------------------------------------------
\section{Preliminaries} \label{sec:preliminaries}

In this section, we first provide the basic notation on square integrable group representations and 
continuous wavelet frames in general and subsequently specify it to \blue{our spherical} setting. 
For more information, we refer to \cite{AAG2000,fuhr2005abstract}.

%---------------------------------------------------------------------------------------------------
\subsection{Continuous wavelet frames on homogeneous spaces}

Let $H$ be a Hilbert space with scalar product $\langle \cdot, \cdot \rangle
$ and norm  $\|\cdot \| = \langle \cdot, \cdot \rangle^\frac12$, and fix
a locally compact space $X$ equipped with a Radon-measure $\nu$. 
A family $\{\eta_x\}_{x\in X}\subset H$ is called a \emph{continuous frame}, 
if for every $ \phi \in H$ the map $x \mapsto \langle \phi , \eta_x \rangle $ is measurable and 
there exist constants $0<\red{c\leq C} < \infty$ such that
\begin{equation}
\red{c}\|\phi \|^2 \leq \int_X |\langle \phi,\eta_x\rangle|^2\,d\nu(x)\leq
\red{C}\|\phi\|^2\label{eq:6}
\end{equation}
for all $\phi\in H$.
For a continuous frame $\{\eta_x\}_{x\in X}$, the corresponding
\emph{frame operator} $A$ is given 
in its weak formulation by
$$\langle A \phi,\psi\rangle 
=
\int_X \langle \phi,\eta_x \rangle \langle \eta_x, \psi\rangle \,d\nu(x), \quad \phi,\psi\in H\,,
$$
and condition~\eqref{eq:6} is equivalent to the fact that the frame
operator $A$ is bounded and boundedly invertible. 

The prototype of \blue{a} continuous frame is given by a square-integrable
representation of a locally compact group $G$ with left Haar measure
$\mu$. Let $U$  be   a \emph{continuous unitary representation} of $G$
in $H$, {\em i.e.}  $U$ is a mapping from 
$G$ into the space $\mathcal U(H)$ of unitary operators on $H$ fulfilling
$U(g g') = U(g) U(g')$ for all $g,g' \in G$, and
the function $g \mapsto \langle \phi, U(g) \psi \rangle$ must be continuous for every $\phi,\psi \in H$.
In what follows, we will only consider continuous representations and
omit the word 'continuous' in this context.  
An irreducible representation  is
called \emph{square integrable}, 
if there exists a vector $\eta \in H \setminus\{0\}$ such that
\begin{equation}\label{eq:admiss_group}
\int_G\bigl|\langle\eta,U(g)\eta\rangle\bigr|^2\,d\mu(g)<\infty.
\end{equation}
In this case, the  vector $\eta$ is called \textit{admissible}.
With the choice $\nu=\mu$, general
results from representation theory
\cite{duflo1976regular} guarantee that
the  family $\{U(g)\eta\}_{g\in G}$ is a
continuous frame, and  the
corresponding frame operator $A_\eta$ with
\begin{equation} \label{frame_op}
\langle A_\eta\phi,\psi\rangle
= 
\int_G\langle\phi,U(g)\eta\rangle\langle U(g)\eta,\psi\rangle\,d\mu(g) , \quad \phi,\psi \in H,
\end{equation}
is in fact  a multiple of the identity.  

In this paper, we are concerned with a homogeneous space $X = G/N$, 
where $N$ is a closed subgroup of $G$ 
instead of the whole group.  We fix
on $X$ a strongly quasi-invariant
measure $\nu$, see~\eqref{quasi} and 
\cite{folland2016course} for the
general theory. Since in general a representation is not directly defined on $X$, 
we need to introduce a (measurable) section $\sigma: X \rightarrow G$, i.e. a map which for each $x\in X$ assigns an element $\sigma(x)\in G$ such that $\sigma(x)$ belongs to the coset $x$.
We then call a representation $U$ of $G$ in $H$ \emph{square-integrable modulo} $(N,\sigma)$, if there exists
a vector $\eta\in H \backslash\{0\}$ such that 
the family
$\{U(\sigma(x))\eta\}_{x\in X}$ is a
continuous frame, {\em i.e.}  such that the operator $A_\eta$ weakly defined via
	\begin{equation}\label{eq:def-cont-frame}
		\langle A_\eta \phi,\psi\rangle
		=\int_X \langle \phi,  U\left(\sigma(x) \right)\eta \rangle \langle U\left(\sigma(x) \right)\eta,\psi\rangle\,d\nu(x)\,,
		\quad\phi,\psi\in H,
	\end{equation}
is bounded and boundedly
invertible. We call the
corresponding vector $\eta\in
H\backslash\{0\}$ \emph{admissible}.

Notice that the definition of
admissibility is often given in a
different way in the literature, as pointed out in the following remark.

\begin{rem}
\red{We observe that in \cite{AV99} and in
\cite{DST2004} a vector $\eta\in
H\backslash\{0\}$ is called \textrm{admissible} if 
\begin{equation}
\int_X\bigl|\langle\phi,U(\sigma(x))\eta\rangle\bigr|^2\,d\nu(x)<\infty\label{eq:9}
\end{equation}
for all $\phi\in H$.
If $X=G$ and the
representation $U$ is irreducible,
then the upper and lower bounds follow from~\eqref{eq:admiss_group}.
However, if $X=G/N$ with $N$
non-trivial, ~\eqref{eq:admiss_group} does not imply~\eqref{eq:9},
which in general is not sufficient to ensure the lower bound, see
\cite{fuhr2005abstract} for a
complete discussion. }
\end{rem}

%----------------------------------------------------------------------
\subsection{Group theoretic aspects} \label{ssec:rep-sphere}
%----------------------------------------------------------------------
 \red{We briefly recall the group theoretical
construction, introduced in~\cite{AV99}, of wavelets on the two dimensional sphere
$$
\mathbb S^2 := \{ \omega\in\R^3 : \omega^\tT\omega=1 \} =\{(\sin \theta
\cos \varphi,\sin \theta \sin\varphi, 
\cos \theta)^\tT: \varphi \in [0,2\pi), \theta \in [0,\pi]\}.  
$$
There is no harm in  regarding $\omega$ and $(\theta,\varphi)$ both as points on the sphere and as points
in $\mathbb R^3$ and we shall do so.  We equip $\mathbb S^2$ with the Riemannian
surface element  $d\Sigma(\omega) = \sin \theta d \theta d \varphi$.
The Hilbert space $H$  we are interested in is $L_2(\mathbb S^2) =
L_2(\mathbb S^2,\Sigma)$ and we denote the  inner product by
$\langle\cdot,\cdot\rangle$ and the norm by 
$\|\cdot\| = \langle \cdot, \cdot\rangle^\frac12$. 

The group $G=\SO(3,1)_0$ is the connected component of the identity of the \emph{Lorentz group}
\[ \Oo(3,1) = \{ g \in \mathbb R^{4,4}: g^\tT I_{3,1}  g = I_{3,1}\},\] 
where $I_{3,1}=\operatorname{diag}(1,1,1,-1)$. Its \emph{Iwasawa
  decomposition} reads as 
$$\SO(3,1)_0=KAN,$$
where $K$ is the maximal compact subgroup of \emph{rotations }
\[
K=\left\{
\begin{bmatrix}
  \gamma & 0\\
  0 & 1
\end{bmatrix} : \gamma \in \SO(3) \right\} \sim \SO(3)= \{\gamma  \in
\mathbb R^{3,3}: 
\gamma^\tT \gamma = I_3, \det \gamma = 1\}  ,
  \]
the factor $A$ is the Abelian subgroup of \emph{dilations}, 
\[
A = \left\{
\begin{bmatrix}
I_2 & 0 & 0 \\
  0  & \frac{a+a^{-1}}{2} &\frac{a-a^{-1}}{2} \\
  0  & \frac{a-a^{-1}}{2}& \frac{a+a^{-1}}{2}
\end{bmatrix} :  a \in \R_*^+ \right\} \sim \R_*^+,
 \]
 and $N $ is the nilpotent (in fact Abelian) subgroup of \emph{translations}
\[
N = \left\{
\begin{bmatrix}
   I_2 & -b  &   b \\
  b^\tT  &  1- \frac{|b|^2}{2}&  \frac{|b|^2}{2}\\
  b^\tT  & -\frac{|b|^2}{2} & 1+\frac{|b|^2}{2}
\end{bmatrix} :  b \in  \R^2 \right\} \sim \R^2 .
\]
From now on, we identify $K$, $A$ and $N$ with $\SO(3)$, $\R_*^+$ and
$\R^2$, respectively. The minimal parabolic subgroup of $\SO(3,1)_0$ is
    \[ P= MAN,\]
 where $M$ is the centralizer of $A$ in $K$, namely
\[
M=\left\{
\begin{bmatrix}
u&0\\
 0& I_2
\end{bmatrix} :u\in \SO(2)\right\} \sim \SO(2) = \{ u \in
\mathbb R^{2,2}: 
u^\tT u = I_2, \det u = 1\}.  
\]
Since 
\[
\SO(3,1)_0/P = KAN/MAN\sim K/M=\SO(3)/\SO(2)\sim \mathbb S^2,
  \]
the group $\SO(3,1)_0$ acts transitively on $\mathbb S^2$ and we denote
the corresponding action  by
\[
\SO(3,1)_0\times \mathbb S^2 \ni (g,\omega)\mapsto g\omega \in \mathbb S^2 .
  \]
  In particular, for all rotations $\gamma\in \SO(3)$, we have
  \begin{alignat}{1}
    \label{eq:4}
    \gamma \omega & = \gamma .\omega ,
  \end{alignat}
  where $\gamma .\omega $ denotes the action of the matrix $\gamma$ on
  the vector $\omega$, and for all  dilations $a\in \R_*^+$,
  \[
a \omega = (\theta_a ,\varphi)=:\omega_a, \qquad  \tan(\tfrac{1}{2}\theta_a)=a\tan(\tfrac{1}{2}\theta).
\]
The measure $d\Sigma$ is quasi-invariant with respect to the action of
$\SO(3,1)_0$. This means for all $g\in\SO(3,1)$ that
  \begin{equation}
\int_{\mathbb S^2} f(g\omega)\, d\Sigma(\omega)= \int_{\mathbb S^2} \kappa(g,\omega)
f(\omega)\, d\Sigma(\omega)
,\label{quasi}
\end{equation}
where $f:\mathbb S^2\to\C$ is such that one of the two sides, hence both, is
  finite, and   $\kappa$ is a function that enjoys the cocycle property
\begin{equation}\label{eq-cocycle}
	\kappa(g_1 g_2,\omega)=\kappa(g_1,\omega)\kappa(g_2,g_1^{-1}\omega),\quad g_1,g_2\in \SO(3,1), \omega\in \mathbb S^2 .
      \end{equation}     
In particular, for all rotations $\gamma\in\SO(3)$, 
      \[ \kappa(\gamma,\omega) =     1, \quad \omega \in \mathbb S^2,\]
and for all dilations $a\in\R_*^+$, 
\[
\kappa(a,\omega) = \frac{4a^2}{[(a^2-1)\cos\theta+(a^2+1)]^2}, \quad \omega = (\theta,\varphi) \in \mathbb S^2.
  \]
As a particular instance of the theory of parabolic induction
\cite{knapp2013lie}, there is a   natural irreducible unitary
representation of $\SO(3,1)_0$ acting on $L_2(\mathbb S^2)$ as
\begin{equation}\label{eq-rep-hom}
\bigl[U(g)f\bigr](\omega)=\kappa(g,\omega)^{1/2}f(g^{-1}\omega),\quad g\in \SO(3,1), \, f\in L_2( \mathbb S^2).
\end{equation}
The above representation is not square integrable. However, as
first stated in \cite{AV99}, it is
square-integrable modulo $(N,\sigma)$, where  the  homogeneous space is
\[
  X:=\SO(3,1)_0/N\sim \SO(3)\times\R_*^+,
  \]
and  the section $\sigma: X\to \SO(3,1)_0$ is 
$$\sigma(\gamma,a)=\gamma a,\quad \gamma\in \SO(3), a\in A\,.$$
Note that  $X$ admits an $\SO(3,1)$-invariant measure $\nu$ given by
$$d\nu(\gamma,a)=\frac{d\mu(\gamma)\,da}{a^3},$$
where  $\mu$ is the Haar measure of $\SO(3)$ normalized in such a way
that $\mu(\SO(3))=8\pi^2$ and $da$ is the Lebesgue measure on the real
line. 

As shown in \cite[Prop. 3.2]{AV99},  the representation
\eqref{eq-rep-hom} on the \blue{homogeneous} space $X$ factorizes as 
	$$U(\sigma(\gamma,a))=\lambda(\gamma)D_a,$$
where  $\lambda$ is the quasi-regular representation of $\SO(3)$
acting on  $L_2(\mathbb S^2)$ given for $\gamma\in \SO(3)$ by
\begin{equation}\label{prop-rep-factorized}
\bigl[\lambda(\gamma)f\bigr](\omega)=f(\gamma^{-1}\omega),\quad
\omega\in\mathbb S^2,\, f\in L_2(\mathbb S^2),
\end{equation}
and $a\mapsto D_a$ is the unitary representation of $\R_*^+$ acting on
$L_2(\mathbb S^2)$ via dilation operators, namely, for any $a\in\R_*^+$,  
	$$(D_a f)(\omega) :=\kappa(a,\theta)^{1/2} f(\omega_{1/a}),
        \quad \omega\in\mathbb S^2,\, f\in L_2(\mathbb S^2) .$$

In \cite{AV99} the continuous wavelets on the sphere $\mathbb S^2 $ are defined as
the family
\[
\left\{U(\sigma(x))\eta\right\}_{x\in\SO(3)\times\R_*^+},
\]
where $\eta\in L^2(\mathbb S^2)$ is a suitable function.  Theorem~\ref{thm:main}
below provides necessary and sufficient conditions on the vector
$\eta$ such that
$\left\{U(\sigma(x))\eta\right\}_{x\in\SO(3)\times\R_*^+}$ is a
continuous frame, {\em i.e.} it satisfies~\eqref{eq:6}.   We refer to
\cite{AV99} for a geometrical interpretation of $U(\sigma(x))$ in
terms of stereographic \blue{projections}. 

The above result was first stated in~\cite{AV99} with a sketch of the
proof. In  Section~\ref{NSconditions}, we provide a complete proof, filling in  details that are missing
in \cite{AV99}. To this end,  we denote the spherical harmonics  by
$$
Y_\ell^m (\theta,\varphi) = 
	\sqrt{\frac{2\ell+1}{4\pi}\frac{(\ell-m)!}{(\ell+m)!}}
	P_{\ell}^m(\cos\theta)\mathrm{e}^{\mathrm{i} m\varphi}, \quad |m|\le \ell,\, \ell \in \mathbb N_0,
        $$
where $P_{\ell}^m$ denote the  associated Legendre polynomials. We
recall  that        
        $\{
Y_\ell^m\}_{\ell \in \N_0, |m|\leq \ell}$ 
form an orthonormal basis of $L_2(\mathbb S^2)$ and that for each $\ell\in \N_0$ 
the orthogonal projection  onto the spherical zone
\[\mathcal{Y}_\ell=\Span\{Y_\ell^m:|m|\leq\ell\}\]
is given by
\begin{equation} \label{proj}
\Pi_\ell f(\omega) = \sum_{|m|\leq \ell} \langle f , Y_\ell^m \rangle
Y_\ell^m (\omega) 
=
(2\ell + 1) \int_{\mathbb S^2} P_\ell (\omega\cdot\omega') f(\omega') d\Sigma(\omega'),\qquad\omega\in \mathbb S^2,
\end{equation}
where $P_\ell = P_\ell^0$ is the Legendre polynomials of degree $\ell \in \mathbb N_0$.
\blue{Note that each $\mathcal Y_\ell$ is an invariant subspace of dimension $2\ell+1$ for the representation $\lambda$ of $\SO(3)$, and that the restriction of $\lambda$
to $\mathcal Y_\ell$  is irreducible. 
}}

\section{Main result} \label{sec:main}
%--------------------------------------------------------------------------------------------------
The main results of this paper  are the following necessary and sufficient conditions on a function
$\eta \in L_2(\mathbb S^2)$ to be admissible.

\red{\begin{thm}[Main Result] \label{thm:main}
  Fix   $\eta\in L_2(\mathbb S^2)$ such that
\begin{alignat}{1} \label{eq:7}
    & \int_{\mathbb S^2} |\eta(\theta , \varphi)|\frac{1}{1 + \cos\theta}
    \, d\Sigma(\omega) < +\infty\, ,\\
    & \int_{\mathbb S^2} |\eta(\theta , \varphi)|\frac{
      \tan^2(\tfrac\theta2)}{1 + \cos\theta}    \,d\Sigma(\omega)
    < +\infty \, ,\label{eq:8} \\
    & \esssup_{(\theta,\varphi)\in\mathbb S^2}\left(
    |\eta(\theta,\varphi)|\frac{1+\tan^2(\tfrac\theta2)}{1-\cos\theta}\right)<+\infty.
  \, \label{eta_tilde}  
\end{alignat}
Then there is a constant $B>0$ such that
  \[
\int_X\bigl
|\bigl\langle\phi,U(\sigma(x))\eta\bigr\rangle\bigr|^2\,d\nu(x) \leq  B \|\phi\|^2
    \]
 for all $\phi\in L^2(\mathbb S^2)$ if and only if
		\begin{equation}
			\label{eq:cond-eta}
			\int_{\mathbb S^2} \eta(\theta , \varphi) \frac{1}{1 + \cos\theta}\, d\Sigma(\theta,\phi)  = 0\,.
                      \end{equation}
Under this condition, the frame operator
$A_\eta: L_2(\mathbb S^2)\to L_2(\mathbb S^2)$ defined by
		\begin{equation} \label{eq:2} 
			\langle A_\eta\phi,\psi\rangle
		=\int_X\bigl\langle\phi,U(\sigma(x))\eta\bigr\rangle
			\langle U(\sigma(x))\eta,\psi\rangle\,d\nu(x) 
		\end{equation}
is bounded, and it is boundedly invertible if and only if the function
\begin{equation}
  \label{eq:11}
  \theta \mapsto  \int_0^{2\pi}\eta(\theta,\varphi)\,d\varphi \neq 0 \, .
\end{equation}
\end{thm}
The proof of the theorem is given in the rest of the paper, but
  we  first add a few comments and  some notation. For $\eta\in
  L_2(\mathbb S^2)$, set 
  \begin{alignat}{2}
   \eta^{[1]}(\theta ,\varphi) & := \frac{\eta(\theta , \varphi)}{1 +
     \cos\theta}, \quad &&\quad	(\theta,\varphi)\in \mathbb
   S^2, \label{def:eta-prime}\\ 
	\eta^{[2]}(\theta ,\varphi) & := \eta^{[1]}(\theta
        ,\varphi)\tan^2(\tfrac\theta2), \quad && \quad 
	(\theta,\varphi)\in \mathbb S^2, \label{def:eta-second}\\
      \widetilde\eta(\theta)
      &:=\frac{1}{2\pi}\int_0^{2\pi}\eta(\theta,\varphi)\,d\varphi,
      \quad&&\quad \theta\in[0,\pi]\,.  \label{def:eta-tilde}
    \end{alignat}
Conditions~\eqref{eq:7} and~\eqref{eq:8} state that $ \eta^{[1]}$ and
$ \eta^{[2]}$ are integrable functions, respectively,
and~\eqref{eq:11} means that $\widetilde\eta\neq 0$.  It is
clear that if $\eta$ is an axisymmetric function,
\emph{i.e.}, $\eta$ is independent of  the longitude $\varphi$, then
$\widetilde\eta=\eta$ and~\eqref{eq:11} simply states that $\eta$ is a
non-zero vector. Under these
assumptions,  if $\eta$ is a continuous function with support contained
in $(0,\pi)$, conditions \eqref{eq:7},~\eqref{eq:8}
and~\eqref{eta_tilde} hold true. \blue{The following result shows that there is
a rather natural construction for functions $\eta$ fulfilling the
vanishing mean condition \eqref{eq:cond-eta}.} }

\begin{lemma}\label{exist:eta}
	For every $\zeta\in L_2(\mathbb S^2)$ and $\alpha>0$, we have
	$$\int_{\mathbb S^2}\frac{\zeta(\theta,\varphi)}{1+\cos\theta}\,d\Sigma(\omega)
	=\frac1\alpha\int_{\mathbb S^2}\frac{D_\alpha\zeta(\theta,\varphi)}{1+\cos\theta}\,d\Sigma(\omega).$$
	As a consequence, the function
        $\eta=\zeta-\alpha^{-1}D_\alpha\zeta$ fulfills the
        cancellation condition 
	\eqref{eq:cond-eta}.  
\end{lemma}

The proof of the lemma is given at the end of  Section \ref{sec:lower-frame-bound}.

%---------------------------------------------------------------------------
\section{General necessary and sufficient admissibility condition} \label{sec:nec-suff-condition-for-admissibilty}\label{NSconditions}
%---------------------------------------------------------------------------
The following theorem gives a necessary and sufficient condition 
for the frame operator associated to $\{U(\sigma(x))\eta\}_{x\in X}$ 
to be bounded and boundedly invertible. The relation to \cite[Theorem 3.3]{AV99} is explained in a subsequent remark.
We set:
$$\eta_a := D_a \eta\, .$$

\begin{thm}\label{prop-admissible-upper-er}
	\red{Take $\eta\in L_2(\mathbb S^2)$  and set}
	\begin{equation}\label{G_ell}
		G_{\ell}:=
		\frac{1}{2\ell+1}\int_0^\infty
		\langle \Pi_{\ell}\eta_a , \eta_a \rangle \frac{da}{a^3} 
		= \frac{1}{2\ell+1}\int_0^\infty
		\| \Pi_{\ell}\eta_a \|^2 \frac{da}{a^3}\,,
		\qquad\ell\in\N_0.
	\end{equation}
\red{Then} the following conditions are equivalent: 
	\begin{itemize}
		\item[i)]
		There exists a constant $C_\eta>0$ such that for every $\phi\in L_2(\mathbb S^2)$,
		\begin{equation}\label{eq:1}
			\int_X\bigl|\bigl\langle\phi,U(\sigma(x))\eta\bigr\rangle\bigr|^2\,d\nu(x)
			\leq C_\eta\|\phi\|^2.
		\end{equation}
			\item[ii)]
		The sequence $(G_\ell)_{\ell\in\N_0}$ is bounded.
	\end{itemize}
	If one of the above two conditions holds true, then the best constant $C_\eta$ is given by
		$$C_\eta=8\pi^2\sup_{\ell\in\N_0}G_\ell,$$
and the operator $A_\eta: L_2(\mathbb S^2)\to L_2(\mathbb S^2)$ defined by \eqref{eq:2} fulfills
		\begin{align}
			\langle A_\eta\phi,\psi\rangle
			&=\red{8\pi^2} \sum_{\ell\in\N_0} G_\ell \langle \Pi_\ell \phi,\psi\rangle, \label{eq:3}
		\end{align}
		for all $\phi,\psi\in L_2(\mathbb
                S^2)$. \red{Furthermore, $A_\eta$ is boundedly invertible if and only if the
	sequence $(G_\ell)_{\ell\in \N_0}$ is bounded from below by a
        strictly positive constant.  
In this case, we  have
	$$\|A_\eta^{-1}\|=\Bigl(\inf_{\ell\in\N_0}G_\ell\Bigr)^{-1}\,.$$
 \blue{As a consequence}, $\eta$ is an admissible vector if and only if there exist
constants $0<c_\eta\leq C_\eta$ such that for all $\ell\in\N_0$,
\[ c_\eta\leq G_\ell\leq C_\eta\, .\]
}                  
\end{thm}

\begin{proof}
	Standard arguments show that condition~i) is equivalent to assuming the existence of a dense
	subset $\mathcal D\subseteq L_2(\mathbb S^2)$ such that \eqref{eq:1} holds
	true for any $\phi\in \mathcal D$ \cite{fuhr2005abstract}.
	For the  sake of completeness, we outline
	the proof. Define the voice transform 
	$V_\eta: \operatorname{dom} (V_\eta) \subseteq L_2(\mathbb S^2) \to L_2(X,\nu)$ by
	\[
	V_\eta \phi(x) := \langle\phi,U(\sigma(x))\eta \rangle,
	\qquad\text{ for a.e. }x\in X\,,
	\]
	where
	\[\dom(V_\eta)
	:=\{ \phi \in L_2(\mathbb S^2)\mid \langle\phi,U(\sigma(\cdot))\eta \rangle\in L^2(X,\nu)\}.
	\]
	Since $\dom(V_\eta)$ is a vector space containing $\mathcal D$, we see that $V_\eta$ is a
	densely defined operator. A direct computation shows that $V_\eta$ is closed. Thus, the closed
	graph theorem \red{together with~the fact that \eqref{eq:1} holds true for all
          $\phi\in\mathcal D$} implies that $\dom V_\eta = L_2(\mathbb S^2)$ and $V_\eta$ is a bounded operator,
	so that~\eqref{eq:1} holds true for every 	$\phi\in L_2(\mathbb S^2)$.  
	Furthermore, it is straightforward to
	check that $A_\eta=V_\eta^*V_\eta$ is given by~\eqref{eq:2}, so that $A_\eta$ is
	a positive bounded operator by construction. 
	
	Hence, it is sufficient to show that condition ii) is equivalent to the existence of a dense
	subset $\mathcal D\subseteq L_2(\mathbb S^2)$ such that~\eqref{eq:1} holds true for every
	$\phi\in \mathcal D$.
	To this purpose, we define 
	\[\mathcal D :=\Biggl\{ \sum_{\ell=0}^L \phi_\ell \ \Big|\ 
	L\in \N_0,\,\phi_\ell \in\mathcal Y_\ell\Biggr\} ,\]
	which is a dense subspace of $L_2(\mathbb S^2)$.  Fixing $\phi\in\mathcal D$, so that
	$\phi= \sum_{\ell=0}^L \Pi_\ell \phi$,  we obtain
	\begin{align*}
		\int_{\SO(3)}\bigl|\bigl\langle\phi,\lambda(\gamma)D_a\eta\bigr\rangle\bigr|^2\,d\mu(\gamma)
		&= \sum_{\ell,\ell'=0}^L\int_{\SO(3)}
		\bigl\langle\Pi_\ell\phi,\lambda(\gamma)D_a\eta\bigr\rangle
		\overline{\bigl\langle\Pi_{\ell'} \phi,\lambda(\gamma)D_a\eta\bigr\rangle}
		\,d\mu(\gamma)\\
		&= \sum_{\ell,\ell'=0}^L\int_{\SO(3)}
		\bigl\langle\Pi_\ell\phi,\lambda_\ell (\gamma) \Pi_\ell \eta_a\bigr\rangle
		\overline{\bigl\langle\Pi_{\ell'} \phi,
			\lambda_{\ell'}(\gamma)\Pi_{\ell'}\eta_a\bigr\rangle }
		\,d\mu(\gamma)\\
		&=  \sum_{\ell=0}^L \frac{8\pi^2}{2\ell+1}
		\bigl\lVert{\Pi_{\ell}\phi\bigr\rVert}^2\,\bigl\lVert{\Pi_{\ell}\eta_a\bigr\rVert}^2 \,,
	\end{align*}
	where $\lambda_\ell$ is the restriction of the representation $\lambda$ to the
	$\lambda$-invariant subspace $\mathcal Y_\ell=\Pi_\ell L_2(\mathbb S^2)$, and that last equality
	holds true since $(\lambda_\ell)_{\ell\in\N_0}$ is a family of 
	irreducible non equivalent representations of the compact group $\SO(3)$
	and the Schur orthogonality relations give
	\[
	\int_{\SO(3)}  
	\bigl\langle \phi,\lambda_\ell (\gamma) \psi \bigr\rangle\,
	\overline{\bigl\langle\phi',\lambda_{\ell'}(\gamma)\psi' \bigr\rangle} \,d\mu(\gamma)
	=\begin{cases}
		\frac{\mu(\SO(3))}{\operatorname{dim}(\mathcal Y_\ell)}
		\langle \phi,\phi' \rangle \langle \psi',\psi \rangle & \ell=\ell', \\
		0 & \ell\neq \ell'
	\end{cases}
	\]
	for all $\phi,\psi\in\mathcal Y_\ell$, $\phi',\psi'\in\mathcal Y_{\ell'}$. 
	Hence, by \eqref{prop-rep-factorized}  and definition of  $\nu$, we have 
	\begin{align}
		\int_X\bigl|\bigl\langle\phi,U(\sigma(x))\eta\bigr\rangle\bigr|^2\,d\nu(x)
		&=\int_0^\infty\left(\int_{\SO(3)}
		\bigl|\bigl\langle\phi,\lambda(\gamma)D_a\eta\bigr\rangle\bigr|^2\,d\mu(\gamma)\right)
		\,\frac{da}{a^3}\,,\\
		&= \sum_{\ell=0}^L \frac{8\pi^2}{2\ell+1}
		\bigl\lVert \Pi_{\ell}\phi\bigr\rVert^2
		\int_0^\infty \bigl\lVert\Pi_{\ell}\eta_a\bigr\rVert^2 \,\frac{da}{a^3} \label{heavy_change}\\
		&= 8\pi^2\sum_{\ell=0}^L G_\ell \bigl\lVert\Pi_{\ell}\phi\bigr\rVert^2,
	\end{align}
	where in view of Fubini's theorem the left-hand side is finite if and only if each
	$G_\ell$ in the sum is finite. Furthermore, we get
	\begin{align*}
		\sup_{\phi\in\mathcal D,\lVert\phi\rVert=1}
		\int_X\bigl|\bigl\langle\phi,U(\sigma(x))\eta\bigr\rangle\bigr|^2\,d\nu(x)
		&= \sup_{L\in\N_0} \sup_{ \substack{\phi\in \oplus_{\ell=0}^L \mathcal Y_\ell\\
				\lVert\phi\rVert=1}} 8\pi^2\sum_{\ell=0}^L G_\ell
		\bigl\lVert \Pi_{\ell}\phi \bigr\rVert^2 \\
		& = 8\pi^2 \sup_{L\in\N_0} \sup_{\ell\leq L} G_\ell
		= 8\pi^2 \sup_{\ell\in\N_0} G_\ell.
	\end{align*}
	The above equality shows that~\eqref{eq:1} holds true for all $\phi\in\mathcal D$
	if and only if the sequence $(G_\ell)_{\ell\in\N_0} $ is bounded and,
	in such a case,  the best constant is  given by $8\pi^2\sup_{\ell\in\N_0}G_\ell$.
	
	It remains to show the representation \eqref{eq:3}. Since $A_\eta$ is a positive bounded
	operator, it is enough to show~\eqref{eq:3} for \red{$\phi=\psi\in\mathcal D$}. With this choice it follows
	\[
	\langle A_\eta \phi,\phi\rangle
	= \lVert V_\eta\phi\rVert^2
	=\int_X\bigl|\bigl\langle\phi,U(\sigma(x))\eta\bigr\rangle\bigr|^2\,d\nu(x)
	= 8\pi^2\sum_{\ell=0}^L G_\ell \bigl\lVert \Pi_{\ell}\phi \bigr\rVert^2 \,.
	\]
\red{The above equation makes it clear that $A_\eta$ is boundedly
  invertible if and only if the sequence $(G_\ell)_{\ell\in\N_0}$ is
  bounded from below by a positive constant. The last claim
characterizing the admissible vectors is now clear. }
\end{proof}

The identity~\eqref{eq:3} states that $A_\eta$ commutes with the left-regular representation
$\lambda$ of $\SO(3)$, which can be \blue{proven directly} using~\eqref{eq:2}. Moreover, it
immediately becomes clear that the boundedness of $(G_\ell)_{\ell\in \N_0}$ is indeed
necessary and sufficient for the boundedness of $A_\eta$.

\begin{rem} [Relation to {\cite[Theorem 3.3]{AV99}}]
Using our notation, Theorem 3.3 in \cite{AV99} claims that there exists a function $\eta$ such that $(G_\ell)_\ell$ is bounded.
Indeed the authors prove essentially the relations given in 
\red{our} Theorem \ref{prop-admissible-upper-er}
and mention at the end of their proof
that there are clearly many functions $\eta$ that satisfy this condition and that all these functions form a dense set in
$L_2(\mathbb S^2)$. 
However, it took us the entire next section to show that such functions indeed exist, though our conditions \textit{don't} describe a dense subset of $L_2(\mathbb S^2)$.

Moreover, since we do not see immediately how to
verify the integration/summation change in \eqref{heavy_change} for an infinite sum over $\ell$ 
and since the finiteness of the whole expression is not ensured, 
we circumvented this difficulty arguing with the dense subset $\mathcal D$ of $L_2(\mathbb S^2)$. \end{rem}

The identity \eqref{eq:3} does not only demonstrate that boundedness of the sequence
$(G_\ell)_{\ell\in \N_0}$ is equivalent to boundedness of the operator $A_\eta$, 
but it also reveals the criterion for its bounded invertibility.
% \begin{cor}
% 	The operator $A_\eta$, defined in \eqref{eq:2} is boundedly invertible if and only if the
% 	sequence $(G_\ell)_{\ell\in \N_0}$ in \eqref{G_ell} is bounded from below by a strictly positive constant. 
% In this case we  have
% 	$$\|A_\eta^{-1}\|=\Bigl(\inf_{\ell\in\N_0}G_\ell\Bigr)^{-1}\,.$$
% \end{cor}
%In other words, a function $\eta\in L_2(\mathbb S^2)$ is admissible if and only if there 
%exist constants $0<A \leq B < \infty$ such that 
%$$A \leq G_\ell \leq B \qquad\text{for all }\ell\in\N_0\,.$$
In the rest of this paper, we investigate under which conditions on $\eta$ this holds true.

%----------------------------------------------------------------------------------------------
\section{Upper frame bound} \label{asymptotics-for-zonal}
%----------------------------------------------------------------------------------------------
In this section, we deduce the first part of Theorem \ref{thm:main}.
First we  derive  necessary and sufficient conditions on $\eta$ 
such that each $G_\ell$, $\ell \in \mathbb N_0$ is finite.
Then we deduce conditions on $\eta$ such that the 
whole sequence $(G_\ell)_{\ell\in \N_0}$ is bounded from above.
By definition \eqref{G_ell} of $G_\ell$, we will need
asymptotic estimates for the zonal projections $\| \Pi_\ell \eta_a\|^2$ as $a \rightarrow 0$.

The next proposition gives a necessary cancellation condition on $\eta$ such that the numbers 
$G_\ell$, $\ell \in \mathbb N_0$ 
are finite. It is directly related to \cite[Prop. 3.6]{AV99}, 
where the authors  deduced that the boundedness of $(G_\ell)_{\ell\in \N_0}$ implies the cancellation property.
Indeed, just one $G_\ell$, $\ell \in \mathbb N_0$ must be finite to
require this condition. 

\begin{prop}\label{lemma-asymp-1}
  Let $\eta\in L_2(\mathbb S^2)$ \red{satisfy~\eqref{eq:7}}.
  %\red{and $\eta^{[1]}$ as
    % in~\eqref{def:eta-prime}} be such that $\eta^{[1]}$ is integrable. 
	Then 
	\begin{equation}\label{eq-necessary-limit-1}
		\lim_{a\downarrow 0}\frac{\langle \Pi_\ell \eta_a , \eta_a \rangle}{a^2}
		=4(2\ell+1)\left|\int_{\mathbb S^2} \red{\eta(\theta ,
                  \varphi) \frac{1}{1 + \cos\theta}\,
                  d\Sigma(\theta,\phi) }\right|^2\,. 
	\end{equation}
	If for a fixed $\ell \in \mathbb N_0$, 
	the number $G_\ell$ defined by \eqref{G_ell} is finite, then
	\begin{equation}
	\int_{\mathbb S^2} \red{\eta(\theta , \varphi) \frac{1}{1 + \cos\theta}\, d\Sigma(\theta,\phi)} =0.
	\end{equation}
\end{prop}

\begin{proof}
	\red{We first prove} \eqref{eq-necessary-limit-1}. Using the integral form \eqref{proj} of
	$\Pi_\ell$, we get
	\begin{align}
		\frac{\langle \Pi_\ell \eta_a , \eta_a \rangle}{a^2 (2\ell +1)} 
		&= 
		\int_{\mathbb S^2}\int_{ \mathbb S^2} \frac1{a^2} P_{\ell} (\omega\cdot \omega')
		D_a\eta(\omega) \overline{D_a\eta(\omega')}\, d\Sigma(\omega)\, d\Sigma(\omega')\\
		&=  \int_{\mathbb S^2}\int_{ \mathbb S^2}  P_{\ell} (\omega\cdot \omega')
		\frac{\kappa(a,\theta)^{1/2}\kappa(a,\theta')^{1/2}}{a^2}
		\eta(\omega_{1/a}) \overline{\eta(\omega'_{1/a})}\,d\Sigma(\omega) \, d\Sigma(\omega')\\
		&=  \int_{\mathbb S^2}\int_{ \mathbb S^2} P_{\ell} (\omega_a\cdot \omega'_a)
		\frac{\kappa(a^{-1},\theta)^{1/2}\kappa(a^{-1},\theta')^{1/2}}{a^2}
		\eta(\omega) \overline{\eta(\omega')}\,d\Sigma(\omega)\, d\Sigma(\omega')\,. \label{eins}
	\end{align}
	Since
	\begin{equation}\label{eq-aha}
	\red{\kappa(a,\theta)^{1/2}a=\frac{2}{(1-a^{-2})\cos\theta+a^{-2}+1}
	\leq\frac{2}{(1+\cos\theta)},}
	\end{equation}
	we conclude
	\begin{equation}\label{eq-kappa-est}
		\frac{\kappa(a^{-1},\theta)^{1/2}}{a} \leq \frac{2}{1+\cos\theta}.
	\end{equation}
	Recalling further that the Legendre polynomials are uniformly bounded on $[-1,1]$, i.e.,
	$$\|P_\ell\|_{L_\infty([-1,1])} = P_\ell(1) = 1\,,$$
	we obtain \red{with~\eqref{eq:7}} an integrable upper bound for \eqref{eins}.
	Since $\theta_a\to 0$ as $a\to 0$ and
	\[
	\omega_a \cdot \omega'_a
	= \sin\theta_a \sin\theta_a' \cos(\varphi-\varphi') + \cos\theta_a \cos\theta_a',
	\]
	we have that $\lim_{a\downarrow 0} \omega_a \cdot \omega'_a =1$. Moreover, we obtain
	\begin{equation} \label{kappa_2}
	\lim_{a\downarrow 0}  \frac{\kappa(a^{-1},\theta)^{1/2}}{a}
	= \lim_{a\downarrow 0}  \frac{2a^{-2}}{(a^{-2}-1)\cos\theta + (a^{-2}+1)}
	=  \frac{2}{\cos\theta +1}.
	\end{equation}
	Therefore, Lebesgue's dominated convergence theorem yields
	\begin{align*}
		\lim_{a\downarrow 0}\frac{\langle \Pi_\ell \eta_a , \eta_a \rangle}{a^2(2\ell +1)} 
		&= 
		\int_{\mathbb S^2}\int_{ \mathbb S^2} P_{\ell} (1) \eta(\omega) \overline{\eta(\omega')}
		\biggl(\frac{2}{1+\cos\theta}\biggr)^2\,d\Sigma(\omega)\,d\Sigma(\omega')\\
		&=\left|\int_{\mathbb S^2} \eta(\omega) \frac{2}{1+\cos\theta} d\Sigma(\omega)\right|^2.
	\end{align*}
        \red{
If the integral defining $G_\ell$ in \eqref{G_ell} is finite, \blue{then, since
the limit exists, we conclude}}
	\begin{equation*}%\label{eq: limit0}
		\lim_{a\downarrow 0}\frac{\langle \Pi_\ell \eta_a , \eta_a \rangle}{a^2}=0.
		\tag*{\qedhere}
	\end{equation*}
\end{proof}
\red{By the above result it is clear that \eqref{eq:cond-eta} is a
  necessary condition to have a finite upper bound in the frame,
  see~\eqref{eq:1}.}

Next we will use higher order asymptotics to give sufficient conditions for  $G_\ell$, $\ell \in \mathbb N_0$ to be finite and for 
$(G_\ell)_{\ell \in \mathbb N_0}$ to be bounded. We start with an auxiliary relation.

\begin{lemma}\label{prop-second-order-asymp}
  Let $\eta\in L_2(\mathbb S^2)$ \red{satisfy~\eqref{eq:7}
  and~\eqref{eq:8}. Then}
% be such that $\eta^{[1]}$ and $\eta^{[2]}$ are integrable.
	\begin{align*}
		\lim_{a\downarrow 0}
		&\frac{1}{a^4}\Biggl(\frac{\langle\Pi_\ell\eta_a,\eta_a\rangle}{2\ell+1}
		-a^2\biggl|\int_{\mathbb S^2}\frac{2\eta(\omega)}{1+\cos\theta}\,d\Sigma(\omega)\biggr|^2\Biggr)\\
		&=
		8P_\ell'(1)\int_{\mathbb S^2}\int_{\mathbb S^2}
		\frac{\tan^2(\tfrac{\theta}{2})+\tan^2(\tfrac{\theta'}{2})
			-2\tan(\tfrac{\theta}{2})\tan(\tfrac{\theta'}{2})\cos(\varphi-\varphi')}
		{(1+\cos\theta)(1+\cos\theta')}\eta(\omega)\overline{\eta(\omega')}\,d\Sigma(\omega) \,d\Sigma(\omega')\\
		&\qquad -4\int_{\mathbb S^2}\int_{\mathbb S^2}
		\frac{\tan^2(\tfrac{\theta}{2})+\tan^2(\tfrac{\theta'}{2})}{(1+\cos\theta)(1+\cos\theta')}
		\eta(\omega)\overline{\eta(\omega')}\,d\Sigma(\omega)\,d\Sigma(\omega)\,.
	\end{align*}
\end{lemma}

\begin{proof}
	Using the integral form \eqref{proj} of $\Pi_\ell$ again, we obtain
	\begin{align}
		&\frac{1}{a^4}\Biggl(
		\frac{\langle\Pi_\ell\eta_a,\eta_a\rangle}{2\ell+1}
		-a^2\biggl|\int_{\mathbb
           S^2}\frac{2\eta(\omega)}{1+\cos\theta}\,d\Sigma(\omega)\biggr|^2\Biggr)
          \label{eq:12}\\
		&=\int_{\mathbb S^2}\int_{\mathbb S^2}
		\frac{P_{\ell} (\omega_a\cdot \omega_a')-1}{a^2}\cdot
		\frac{\kappa(a^{-1},\theta)^{1/2}\kappa(a^{-1},\theta')^{1/2}}{a^2}
		\eta(\omega) \overline{\eta(\omega')}\,d\Sigma(\omega)\,d\Sigma(\omega')\\
		&\qquad +\int_{\mathbb S^2}\int_{\mathbb S^2}
		\frac{1}{a^2}\left(\frac{\kappa(a^{-1},\theta)^{1/2}\kappa(a^{-1},\theta')^{1/2}}{a^2}
		-\frac{4}{(1+\cos\theta) (1+\cos\theta')}\right)
		\eta(\omega)\overline{\eta(\omega')}\,d\Sigma(\omega)\,
           d\Sigma(\omega') \nonumber\\ 
		&=: I_1+I_2.\nonumber
	\end{align}
	Now we intend to apply Lebesgue's dominated convergence theorem to both integrals.
	
	\textbf{Integral $I_1$:} 
		Since $P_\ell(1) = 1$, there exists a polynomial $Q_\ell$ such that 
	\begin{equation} \label{def:q-ell}
		P_\ell(t) -1= Q_{\ell}(t) (1-t)\,,\quad t\in [-1,1]\,.
	\end{equation}
	Moreover, the Mean Value theorem \red{and the properties of
          Legendre polynomials} yield
	$$
	\|\red{Q_\ell}\|_{L_\infty([-1,1])}  \leq \|P_\ell'\|_{L_\infty([-1,1])}
	= P_\ell'(1) =\frac{\ell(\ell+1)}{2}\,.
	$$
	Thus,
        \begin{equation}
	I_1 = \int_{\mathbb S^2}\int_{\mathbb S^2}Q_{\ell} (\omega_a\cdot \omega_a')  \frac{ (1-\omega_a\cdot \omega_a') }{a^2}\,
		\frac{\kappa(a^{-1},\theta)^{1/2}\kappa(a^{-1},\theta')^{1/2}}{a^2}
		\eta(\omega)
                \overline{\eta(\omega')}\,d\Sigma(\omega)\,d\Sigma(\omega')\,.\label{eq:13}       \end{equation}
	Now we get
	\begin{align*}
	\red{0\leq}\ 1-\omega_a \cdot \omega'_a
		&=1-\bigl(\sin\theta_a\sin\theta_a'\cos(\varphi-\varphi')+\cos\theta_a\cos\theta_a'\bigr)\\
		&=1-\frac{4a^2\tan(\theta/2)\tan(\theta'/2)\cos(\varphi-\varphi')
			+(1-\tan^2(\theta_a/2))(1-\tan^2(\theta'_a/2))}
		{(1+\tan^2(\theta_a/2))(1+\tan^2(\theta'_a/2))}\\
		&=2a^2\frac{\tan^2(\theta/2)+\tan^2(\theta'/2)
			-2\tan(\theta/2)\tan(\theta'/2)\cos(\varphi-\varphi')}
		{(1+\tan^2(\theta_a/2))(1+\tan^2(\theta'_a/2))}\,,
	\end{align*}
	so that we can estimate
	\begin{align}
	\red{0\leq}\	1-\omega_a \cdot \omega'_a
		&\leq 2a^2\Bigl(\tan^2(\theta/2)+\tan^2(\theta'/2)
		-2\tan(\theta/2)\tan(\theta'/2)\cos(\varphi-\varphi')\Bigr)\nonumber\\
		&\leq 2a^2\bigl(\tan(\theta/2)+\tan(\theta'/2)\bigr)^2 \nonumber\\
		&\leq 4a^2\bigl(\tan^2(\theta/2)+\tan^2(\theta'/2)\bigr) \nonumber\\
		&\leq
           4a^2\bigl(1+\tan^2(\theta/2)\bigr)\bigl(1+\tan^2(\theta'/2)\bigr)\,.\label{eq:15}
	\end{align}
	Using also \eqref{eq-kappa-est} and
        \red{conditions~\eqref{eq:7} and~\eqref{eq:8}} we obtain an integrable
	majorant. 
	Finally, \eqref{kappa_2}
	together with
	$$
	\lim_{a\downarrow 0}\frac{1-\omega_a\cdot \omega'_a}{a^2}
	=2\Bigl(\tan^2(\tfrac{\theta}{2})+\tan^2(\tfrac{\theta'}{2})
	-2\tan(\tfrac{\theta}{2})\tan(\tfrac{\theta'}{2})\cos(\varphi-\varphi')\Bigr)
	$$
	and
	$$\lim_{a\downarrow 0}Q_\ell(\omega_a \cdot \omega'_a) = Q_\ell(1) = P_\ell'(1),$$
	yield by Lebesgue's dominated convergence theorem 
	$$\lim_{a\downarrow 0}I_1
	=8P_\ell'(1)\int_{\mathbb S^2}\int_{\mathbb S^2}
	\frac{\tan^2(\tfrac{\theta}{2})+\tan^2(\tfrac{\theta'}{2})
		-2\tan(\tfrac{\theta}{2})\tan(\tfrac{\theta'}{2})\cos(\varphi-\varphi')}
	{(1+\cos\theta)(1+\cos\theta')}\eta(\omega)\overline{\eta(\omega')}\,d\Sigma(\omega)\,d\Sigma(\omega').$$
	
	\textbf{Integral $I_2$:} For the second integral $I_2$ we compute
	\begin{align}
		&\frac{\kappa(a^{-1},\theta)^{1/2}\kappa(a^{-1},\theta')^{1/2}}{a^2}
		-\frac{4}{(1+\cos\theta)(1+\cos\theta')}\nonumber\\
		&=\frac{\kappa(a^{-1},\theta)^{1/2}}{a}
		\left(\frac{\kappa(a^{-1},\theta')^{1/2}}{a} - \frac{2}{1+\cos\theta'}\right)
		+\left(\frac{\kappa(a^{-1},\theta)^{1/2}}{a} - \frac{2}{1+\cos\theta}\right)
		\frac{2}{1+\cos\theta'}\,.\label{eq:18}
	\end{align}
	Then we realize that 
	\begin{align}
		\frac{2}{1+\cos\theta} - \frac{\kappa(a^{-1},\theta)^{1/2}}{a}
		&=  \frac{2}{1+\cos\theta} - \frac{2a^{-2}}{(a^{-2}-1)\cos\theta + (a^{-2}+1)} \nonumber\\
		&= \frac{2}{1+\cos\theta} - \frac{2}{(1-a^2)\cos\theta + (1+a^2)} \nonumber\\
		&=a^2\,\frac{2(-\cos\theta + 1)}{(1+\cos\theta)((1-a^2)\cos\theta+ 1+a^2)}\nonumber\\
		&\leq a^2\frac{2}{1+\cos\theta}\,\frac{1-\cos\theta}{1+\cos\theta}
		=a^2\frac{2}{1+\cos\theta}\tan^2(\tfrac{\theta}{2})\,.\label{eq:21}
	\end{align}
	Now \eqref{eq-kappa-est} and \red{conditions~\eqref{eq:7}
          and~\eqref{eq:8}} imply 
	the existence of an integrable majorant. 
	From the above	calculation we can also deduce
	$$\lim_{a\downarrow 0}\frac{1}{a^2}\biggl(\frac{2}{1+\cos\theta}
	-\frac{\kappa(a^{-1},\theta)^{1/2}}{a}\biggr)
	=\frac{2}{1+\cos\theta}\cdot\frac{1-\cos\theta}{1+\cos\theta}.$$
	Now Lebesgue's dominated convergence theorem yields
	$$\lim_{a\downarrow 0}I_2
	=-4\int_{\mathbb S^2}\int_{\mathbb S^2}\frac{\tan^2(\tfrac{\theta}{2})+\tan^2(\tfrac{\theta'}{2})}
	{(1+\cos\theta)(1+\cos\theta')}
	\eta(\omega)\overline{\eta(\omega')}\,d\Sigma(\omega)\,d\Sigma(\omega')\,.$$
	This completes the proof.
\end{proof}

Now we can prove the desired sufficient condition on $\eta$.

\begin{prop}\label{cor:necessary-condition}
Let $\eta\in L_2(\mathbb S^2)$ be such that \red{
conditions~\eqref{eq:7},~\eqref{eq:8} and~\eqref{eq:cond-eta} hold
true.}
	% If additionally
	% \begin{equation} \label{zz}	
	% \int_{\mathbb S^2} 
	% \eta^{[1]} (\omega) \, d\Sigma(\omega) =0\,,
	% \end{equation}
	Then the numbers $G_\ell$ defined in \eqref{G_ell} 
	are finite. 
\end{prop}

\begin{proof}
  We split the integral \red{
      \begin{equation}
G_\ell=\int_0^{\infty}\frac{\langle\Pi_\ell\eta_a,\eta_a\rangle}{(2\ell+1)a^2}
	\,\frac{da}{a}\label{eq:17}      
      \end{equation}}
according to $\R_*^+=(0,\varepsilon]\cup(\varepsilon,\infty)$, with
	$\varepsilon$ to be chosen later on.
	
	\textbf{Step 1:} We first concentrate on the interval $(0,\varepsilon]$, i.e. we discuss
	$$J_1=\int_0^\varepsilon\frac{\langle\Pi_\ell\eta_a,\eta_a\rangle}{(2\ell+1)a^2}
	\,\frac{da}{a}.$$
\red{As in the previous proof, we start \blue{at}~\eqref{eq:12} assuming
  condition~\eqref{eq:cond-eta}. Taking into 
  account \eqref{eq:13} with the bounds~\eqref{eq-kappa-est} and~\eqref{eq:15}, and~\eqref{eq:18}
  with the bound~\eqref{eq:21}, we get }
	\begin{align*}
		\frac{\langle\Pi_\ell\eta_a,\eta_a\rangle}{(2\ell+1)a^2}
		&\leq 16a^2
		\int_{\mathbb S^2}\int_{\mathbb S^2}\bigl|Q_\ell(\omega_a \cdot \omega'_a)\bigr|\\
		&\qquad\qquad\times\bigl(1+\tan^2(\tfrac{\theta}{2})\bigr)
		\bigl(1+\tan^2(\tfrac{\theta'}{2})\bigr)
		\frac{|\eta(\omega)|\,|\eta(\omega')|}{(1+\cos\theta)(1+\cos\theta')}\,d\Sigma(\omega)\,d\Sigma(\omega')\\
		&\qquad +4a^2\int_{\mathbb S^2}\int_{\mathbb S^2}
		\frac{\tan^2(\tfrac{\theta}{2})+\tan^2(\tfrac{\theta'}{2})}{(1+\cos\theta)(1+\cos\theta')}
		|\eta(\omega)|\,|\eta(\omega')|\,d\Sigma(\omega)\,d\Sigma(\omega')\\
		&\leq 
		4a^2\bigl(4\|Q_\ell\|_{L_\infty[-1,1])} +1\bigr)m_2^2\,,
	\end{align*}
	where $Q_\ell$ is given in \eqref{def:q-ell} and 
	$
	m_2=\int_{\mathbb S^2} \red{|\eta^{[1]}(\omega)| + |\eta^{[2]}(\omega)|}\,d\Sigma(\omega)
	$, \red{where $\eta^{[i]}$ are defined
          in~\eqref{def:eta-prime} and~\eqref{def:eta-second}, and
          $m_2$ is finite by assumptions~\eqref{eq:7} and~\eqref{eq:8}.}
	Thus, a subsequent integration w.r.t. $\frac{da}{a}$ yields
	$$J_1\leq 2m_2^2\varepsilon^2\bigl(4\|Q_\ell\|_{L_\infty[-1,1])}+1\bigr)
	=2m_2^2\varepsilon^2\bigl(2\ell(\ell+1)+1\bigr)\,.$$
	
	\textbf{Step 2:} It remains to deal with the integration over $[\varepsilon,\infty)$. But here we
	get the desired estimate more directly: Since $\Pi_\ell$ is an orthogonal projection
	and $D_a$ a unitary operator, we find
	$$\langle\Pi_\ell\eta_a,\eta_a\rangle
	=\|\Pi_\ell D_a\eta\|^2
	\leq\|D_a\eta\|^2=\|\eta\|^2,$$
	and consequently
        \begin{equation}
\frac{1}{2\ell +1}\int_{\varepsilon}^{\infty}
		\langle \Pi_{\ell} \eta_a , \eta_a \rangle\,\frac{da}{a^3}
		\leq\frac{\|\eta\|^2}{2\ell +1}
		\int_{\varepsilon}^{\infty} \frac{da}{a^3}
		=
                \frac{\|\eta\|^2}{2\ell+1}\frac{\varepsilon^{-2}}{2}.\label{eq:16}             
              \end{equation}
	This completes the proof.
\end{proof}

Another slight variation of the above arguments together with fixing a particular choice for the parameter $\varepsilon$ leads to a significant improvement of the $\ell$-dependence, at the cost of \blue{an additional} assumption on $\eta$.

\begin{thm}\label{thm:upper}
  Let $\eta\in L_2(\mathbb S^2)$ be such that \red{
conditions~\eqref{eq:7},~\eqref{eq:8},~\eqref{eta_tilde} 
and~\eqref{eq:cond-eta} hold true.} Then the sequence $(G_\ell)_\ell$ is bounded.
\end{thm}

\begin{proof}
	In order to show that $(G_\ell)_{\ell}$ is bounded, we again
        split the integral~\red{\eqref{eq:17}}    according to
	$\R_*^+=(0,\varepsilon]\cup(\varepsilon,\infty)$. Fixing the choice
	$\varepsilon^{-2}=2(2\ell+1)$, from the estimate~\red{\eqref{eq:16}} it immediately becomes clear that the integral
	over $(\varepsilon,\infty)$ is uniformly bounded by $\|\eta \|^2$. 
	
	Thus it remains to consider the integral over
        $(0,\varepsilon]$. \red{As in the previous proof, we start at~\eqref{eq:12} assuming
  condition~\eqref{eq:cond-eta}. Taking into 
  account \eqref{eq:13} with the bound~\eqref{eq:15}, and~\eqref{eq:18}
  with the bound~\eqref{eq:21}, we get }
	\begin{align}
		\frac{\langle\Pi_\ell\eta_a,\eta_a\rangle}{(2\ell+1)a^2}
		&\leq 4a^2\int_{\mathbb S^2}\int_{\mathbb S^2}\bigl|Q_\ell(\omega_a\cdot \omega'_a)\bigr|
		\bigl(1+\tan^2(\tfrac{\theta}{2})\bigr)\bigl(1+\tan^2(\tfrac{\theta'}{2})\bigr)
		\nonumber\\
		&\qquad\qquad\times\frac{\kappa(a^{-1},\theta)^{1/2}\kappa(a^{-1},\theta')^{1/2}}{a^2}
		|\eta(\omega)|\,|\eta(\omega')|\,d\Sigma(\omega)\,d\Sigma(\omega')
		\label{eq:est-1}\\
		&\qquad +4a^2\int_{\mathbb S^2}\int_{\mathbb S^2}
		\frac{\tan^2(\tfrac{\theta}{2})+\tan^2(\tfrac{\theta'}{2})}{(1+\cos\theta)(1+\cos\theta')}
		|\eta(\omega)|\,|\eta(\omega')|\,d\Sigma(\omega)\,d\Sigma(\omega')\,.
		\label{eq:est-2}
	\end{align}
	Setting
	$$C_1 :=\int_{\mathbb S^2}\frac{|\eta(\omega)|}{1+\cos\theta}
	\bigl(1+\tan^2(\tfrac\theta2)\bigr)\,d\Sigma(\omega)\,,$$
\red{which is finite by assumptions~\eqref{eq:7} and~\eqref{eq:8},}      
	we can estimate the integral \eqref{eq:est-2} by
	$$4a^2\int_{\mathbb S^2}\int_{\mathbb S^2}
	\frac{\tan^2(\tfrac{\theta}{2})+\tan^2(\tfrac{\theta'}{2})}{(1+\cos\theta)(1+\cos\theta')}
	|\eta(\omega)|\,|\eta(\omega')|\,d\Sigma(\omega)\,d\Sigma(\omega')\leq 4a^2 C_1^2\,.$$
	Next, we define
	$$C_2 :=\esssup_{\omega\in \mathbb S^2}
	\frac{|\eta(\omega)|}{1-\cos\theta}\bigl(1+\tan^2(\tfrac\theta2)\bigr)\,$$
\red{which is finite by assumption~\eqref{eta_tilde}.}       
After rewriting \eqref{eq:est-1} as
	\begin{align*}
		4a^2
		&\int_{\mathbb S^2}\int_{\mathbb S^2}\bigl|Q_\ell(\omega\cdot \omega')\bigr|
		\bigl(1+\tan^2(\tfrac{1}{2}\theta_{1/a})\bigr)\bigl(1+\tan^2(\tfrac{1}{2}\theta'_{1/a})\bigr)
		\nonumber\\
		&\qquad\qquad\times\frac{\kappa(a,\theta)^{1/2}\kappa(a,\theta')^{1/2}}{a^2}
		|\eta(\omega_{1/a})|\,|\eta(\omega'_{1/a})|\,d\Sigma(\omega)\,d\Sigma(\omega')\,,
	\end{align*}
	and applying $\frac{\kappa(a,\theta)^{1/2}}{a}\leq\frac{2}{1-\cos\theta}$,  we can
	estimate this by
	\begin{align*}
		16a^2
		&\biggl(\esssup_{\omega=(\theta,\varphi)\in \mathbb S^2}
		\frac{|\eta(\omega_{1/a})|}{1-\cos\theta}\bigl(1+\tan^2(\tfrac{1}{2}\theta_{1/a})\bigr)\biggr)^2
			\int_{\mathbb S^2}\int_{\mathbb S^2}\bigl|Q_\ell(\omega\cdot \omega')\bigr|
			\,d\Sigma(\omega)\,d\Sigma(\omega')\\
		&= 16a^2\biggl(\esssup_{\omega=(\theta,\varphi)\in \mathbb S^2}
			\frac{|\eta(\omega)|}{1-\cos\theta_a}\bigl(1+\tan^2(\tfrac{\theta}{2})\bigr)\biggr)^2
			\int_{\mathbb S^2}\int_{\mathbb S^2}\bigl|Q_\ell(\omega\cdot \omega')\bigr|
			\,d\Sigma(\omega)\,d\Sigma(\omega')\\
		&\leq
		16a^2 C_2^2\int_{\mathbb S^2}\int_{\mathbb S^2}|Q_{\ell} (\omega\cdot \omega')|\,d\Sigma(\omega)\,d\Sigma(\omega')\,,
	\end{align*}
	where at the end we used $a\leq\varepsilon<1$, which implies $\theta_a\leq\theta$ and thus
	$1-\cos\theta\geq 1-\cos\theta_a$.
	To calculate the last integral, we observe for $Q_\ell$ defined in \eqref{def:q-ell} and for fixed	$\omega=(\theta,\varphi)$ that
	\begin{align*}
		\int_{\mathbb S^2}|Q_\ell(\omega\cdot\omega')|\,d\Sigma(\omega')
		&=\int_{\mathbb S^2}|Q_\ell(R_\omega\omega\cdot R_\omega\omega')|\,d\Sigma(\omega')
		=\int_{\mathbb S^2}|Q_\ell(e_z\cdot\omega'')|\,d\Sigma(\omega'')\\
		&=\int_0^{2\pi}\int_0^\pi|Q_\ell(\cos\theta'')|\,\sin\theta''d\theta''\,d\varphi''\\
		&=2\pi\int_{-1}^1|Q_\ell(s)|\,ds=2\pi \|Q_\ell\|_{L_1([-1,1])} \,,
	\end{align*}
	where the rotation $R_\omega$ maps $\omega$ onto the north pole $e_z$, and  we used
	the substitution $s=\cos\theta''$. 
	We conclude
	$$
	\int_{\mathbb S^2}\int_{ \mathbb S^2}|Q_{\ell} (\omega\cdot \omega')|\,d\Sigma(\omega')\,d\Sigma(\omega)
	=8\pi^2\|Q_\ell\|_{L_1([-1,1])}.
	$$
	Arranging everything together, we arrive at 
	\begin{align*}
		\frac{\langle\Pi_\ell\eta_a,\eta_a\rangle}{(2\ell+1)a^2}
		&\leq 2C_3 a^2\bigl(1+\|Q_\ell\|_{L_1([-1,1])}\bigr)\,,
	\end{align*}
\red{where $C_3$ is a suitable constant,}  which implies the estimate
	\begin{align}\label{eq:bounded-J1}
		\int_0^\varepsilon\frac{\langle\Pi_\ell\eta_a,\eta_a\rangle}{(2\ell+1)a^2}\,\frac{da}{a}
		\leq C_3\varepsilon^2\bigl(1+\|Q_\ell\|_{L_1([-1,1])}\bigr)\,.
	\end{align}
	Recalling $\varepsilon^2\sim\ell^{-1}$, it remains to show that
	$\|Q_\ell\|_{L_1([-1,1])}=\mathcal{O}(\ell)$. 
	Based
	on the Mean Value Theorem applied to $P_\ell$, we have for every $x\in [0,1]$ that
	$$ |Q_\ell (x)| =  \frac{1-P_\ell (x)}{1-x}
	\leq \| P'_\ell\|_{L_\infty([0,1])} = \frac{\ell(\ell+1)}{2}. $$
	Thus, we obtain
	\begin{align*}
		\int_{-1}^{1} \left|Q_\ell (x) \right| dx
		&= \int_{-1}^{\frac{\ell-1}{\ell}} \left| \frac{1-P_\ell (x)}{1-x} \right| dx
		+ \int_{\frac{\ell-1}{\ell}}^{1} |Q_\ell (x) | dx\\
		&\leq \frac{2}{1-\frac{\ell-1}{\ell}}\Bigl(\frac{\ell-1}{\ell}+1\Bigr)
		+ \frac{\ell(\ell+1)}{2}\Bigl(1-\frac{\ell-1}{\ell}\Bigr)\\
		&=2(2\ell-1)+\frac{\ell+1}{2}=\frac{9\ell - 3}{2},
	\end{align*}	
	which shows that the right-hand side of \eqref{eq:bounded-J1} is uniformly bounded.
\end{proof}

%-------------------------------------------------------------------------------------
\section{Stereographic projection and moment conditions} \label{ssec:stereo}
%-------------------------------------------------------------------------------------
In this section, we consider an isometric transform of the mother wavelet $\eta$ based on the stereographic projection.
\red{Recalling the definition of $\eta^{[1]}$ and $\eta^{[2]}$ given
  by~\eqref{def:eta-prime} and~\eqref{def:eta-second}}, we reinterpret the cancellation condition \eqref{eq:cond-eta} on $\eta^{[1]}$
as a vanishing mean condition of the transformed $\eta$ 
and integrability condition~\red{\eqref{eq:8}} of $\eta^{[2]}$ as existence condition of the second moment of the transformed $\eta$.
The interesting relation between the cancellation condition and the transformed admissible function was already
discovered by Antoine and \blue{Vandergheynst}~\cite{AV99}.

The stereographic projection allows to map $\mathbb S^2$ 
without the south pole onto the real plane.
In polar coordinates in the plane the stereographic projection and its inverse are given by
$$(\theta,\varphi)\mapsto(\tan(\tfrac{\theta}{2}),\varphi), \quad
(r,\varphi)\mapsto (2\arctan r,\varphi),
$$ 
respectively.
The stereographic projection allows to map functions on the sphere to functions on the real plane and vice versa.
The following lemma  corresponds to
\cite[Lemma 3.5]{AV99} under the correspondence $2\arctan r=\arccos\frac{1-r^2}{1+r^2}$

\begin{lemma}\label{lemma-stereo}
	The mapping 	$\Theta:L_2(\mathbb S^2)\rightarrow L_2(\R_+\times[0,2\pi),r^{-1}dr\,d\varphi)$, defined by 
		$$(\Theta f)(r,\varphi)=\frac{2r}{1+r^2}f\bigl(2\arctan r,\varphi\bigr)$$
	is an isometric isomorphism.
	Moreover, the usual dilation  $D_a^+$ on $\R^2$, which can be written in polar coordinates 
	as $(D_a^+ f)(r,\varphi)=f(r/a,\varphi)$, fulfills the relation
	$\Theta D_a=D_a^+\Theta$.
\end{lemma}

\begin{proof}
	Since $\theta=2\arctan r$ if and only if $ r=\tan(\theta/2)$, we get
	\begin{align*}
		\frac{d\theta}{dr}=\frac{2}{1+r^2},
	\end{align*}
	and furthermore
	$$\sin\theta=\frac{2\tan(\theta/2)}{1+\tan^2(\theta/2)}=\frac{2r}{1+r^2}.$$
	Inserting this, we conclude
	\begin{align*}
		\|\Theta f\|^2_{L_2(\R_+\times[0,2\pi),r^{-1}dr\,d\varphi)}
		&=\int_0^{2\pi}\int_{\R_+}|(\Theta f)(r,\varphi)|^2 r^{-1}\,dr\,d\varphi\\
		&=\int_0^{2\pi}\int_0^\pi\biggl(\frac{2r}{1+r^2}\biggr)^2
		|f(\theta,\varphi)|^2 r^{-1}\frac{1+r^2}{2}	\,d\theta\,d\varphi\\
		&=\int_0^{2\pi}\int_0^\pi|f(\theta,\varphi)|^2 \sin\theta\,d\theta\,d\varphi
		=\|f\|^2.
	\end{align*}
	Thus, $\Theta$ is an isometry, and obviously bijective. For the second part, we
	have on the one hand
	\begin{align*}
		(\Theta D_a f)(r,\varphi)
		&=\Theta\bigl[\kappa(a,\theta)^{1/2}
		f\bigl(2\arctan(\tfrac{1}{a}\tan\tfrac{\theta}{2}),\cdot\bigr)\bigr](r,\varphi)\\
		&=\kappa(a,2\arctan r)^{1/2}\frac{2r}{1+r^2}f\bigl(2\arctan(r/a),\varphi\bigr),
	\end{align*}
	and since  $\cos\theta=\frac{1-r^2}{1+r^2}$ we further get
	$$\kappa(a,2\arctan r)^{1/2}
	=\frac{2a}{(a^2-1)\frac{1-r^2}{1+r^2}+(a^2+1)}
	=\frac{2a(1+r^2)}{2a^2+2r^2}\,.$$
	On the other hand,
	\begin{align*}
		(D_a^+\Theta f)(r,\varphi)
		&=D_a^+\Bigl[\frac{2r}{1+r^2}f\bigl(2\arctan r,\cdot\bigr)\Bigr](r,\varphi)\\
		&=\frac{2ar}{a^2+r^2}f\bigl(2\arctan(r/a),\varphi\bigr)\,,
	\end{align*}
	which finally shows $\Theta D_a=D_a^+\Theta$.
\end{proof}

\begin{rem}
	The substitution $r=\tan(\theta/2)$ also yields the new interpretation of the cancellation condition \eqref{eq:cond-eta} on $\eta^{[1]}$
		and the integrability~\red{\eqref{eq:8}} of $\eta^{[2]}$. First, we obtain
	\begin{align*}
		\int_0^{2\pi}\int_0^\pi\frac{\eta(\theta,\varphi)}{1+\cos\theta}\sin\theta\,d\theta\,d\varphi
		&=\int_0^{2\pi}\int_0^\pi\eta(\theta,\varphi)\tan(\theta/2)\,d\theta\,d\varphi\\
		&=\int_0^{2\pi}\int_0^\infty\eta(\theta,\varphi)\,\frac{2r}{1+r^2}\,dr\,d\varphi\\
		&=\int_0^{2\pi}\int_0^\infty(\Theta\eta)(r,\varphi)\,dr\,d\varphi\,.
	\end{align*}
	Thus, $\eta^{[1]}$ is integrable if $\Theta \eta$ 
	is integrable w.r.t. the Lebesgue measure as a function on $\R_+\times(0,2\pi)$.
	The cancellation condition on $\eta^{[1]}$
	 becomes a 	vanishing mean condition for $\Theta\eta$ w.r.t. the Lebesgue measure $dr\,d\varphi$ on
	$(0,\infty)\times(0,2\pi)$. Note that this is neither the Lebesgue measure on $\R^2$ in polar coordinates, nor
	the measure considered in Lemma \ref{lemma-stereo}. 
	
	Similarly, we can rewrite
	$$\int_{\mathbb S^2}\frac{\eta(\theta,\varphi)}{1+\cos\theta}\tan^2\bigl(\tfrac\theta2\bigr)\,d\Sigma
	=\int_0^{2\pi}\int_0^\infty r^2(\Theta\eta)(r,\varphi)\,dr\,d\varphi\,.$$
	Hence the integrability condition for $\eta^{[2]}$ turns into the existence of the second moment of
	$\Theta\eta$ when considered as a function in	$L_2((0,\infty)\times(0,2\pi),dr\,d\varphi)$.
	
	Finally, the boundedness of $\frac{\eta}{1-\cos\theta}(1+\tan^2(\frac\theta2))$ 
	transfers to the decay of $\Theta\eta$ as $r\to 0$ and $r\to\infty$, respectively, in view of
	$$\Theta\Big(\frac{\eta}{1-\cos\theta}\bigl(1+\tan^2(\tfrac\theta2)\Big)
	=\frac{(1+r^2)^2}{2r^2}\Theta\eta\,.$$
Remarkably, the conditions for $\eta$ to be admissible resemble conditions 
for the Fourier transform of $\Theta\eta$ to be a wavelet in $\R^2$ -- weighted integrability and decay towards the origin ($\longleftrightarrow$ smoothness and vanishing moments, respectively, for the wavelet).
\end{rem}

%------------------------------------------------------------------------------------
\section{Lower frame bound}\label{sec:lower-frame-bound}
%------------------------------------------------------------------------------------
In this section, we are concerned with sufficient conditions for the invertibility of $A_\eta$ and boundedness of
$A_\eta^{-1}$. 
In Proposition \ref{prop-strictly-positive}, we  prove that the numbers $G_\ell$, $\ell \in \mathbb N_0$, are positive
if $\tilde \eta$ defined in \eqref{eta_tilde} does not vanish, 
and in Theorem \ref{thm:lower} that these numbers are indeed uniformly bounded away from zero under \red{Assumption~\eqref{eq:7}.}
Since \cite[Prop. 3.4]{AV99} claims the same, we start with a remark pointing out the differences.

\begin{rem}[Relation to {\cite[Proposition 3.4]{AV99}}]
The first part of the proof of Proposition 3.4 in \cite{AV99} is devoted to show that $G_\ell >0$ for every $\ell \in \mathbb N_0$.
It ends with the claim that the fact that the convolution of some functions 
vanishes for all parameters $a >0$
implies that one of the functions is zero.
However, the convolution of two functions vanishes if their respective Fourier transforms have disjoint supports.
Therefore, we invest some work in proving $G_\ell >0$ in the next proposition.

The second part of the proof in \cite{AV99} deals with the uniform boundedness of $(G_\ell)_\ell$ away from zero.
At the end of the proof, it is stated that 'the only contribution of the integral over $\theta$ comes from the region
$a \sim 1/\ell$' \blue{-- that intuition is correct.}
However, with increasing $\ell$ these regions become smaller and smaller, so that we were unfortunately not able to follow the subsequent
argument. We will therefore address the issue again in the subsequent Theorem \ref{thm:lower}.
\end{rem}

\begin{prop}\label{prop-strictly-positive}
	For every $\eta\in L_2(\mathbb S^2)$ satisfying~\eqref{eq:cond-eta} 
	the numbers $G_\ell$, $\ell\in\N_0$ are strictly positive.
\end{prop}

\begin{proof}
	\textbf{Step 1:} 
	Let $\ell\in\N_0$ be fixed. By definition of the zonal projections, we have
	\begin{equation} \label{fourier}
	\|\Pi_\ell\eta_a\|^2=\sum_{|m|\leq\ell}|\langle Y_\ell^n, \eta_a \rangle |^2 = \sum_{|m|\leq\ell}|\widehat{\eta}_a(\ell,m)|^2,
	\end{equation} 
	where $\widehat{\eta}_a(\ell,m):= 
	\langle \eta_a , Y_\ell^m \rangle$. 
	Thus it suffices to show that
	$|\widehat{\eta}_a(\ell,0)|^2>0$. 
	Note that in the
	particular case of an axisymmetric function  $\eta$, i.e., 
	$\eta$ is independent of the longitude $\varphi$, 
	we have actually 	$|\widehat{\eta}_a(\ell,m)|=0$ for all $m\neq 0$. 
	Using Lemma	\ref{lemma-stereo} we rewrite the Fourier coefficients as
	\begin{align*}
		\widehat{\eta_a}(\ell,0)
		&=\langle D_a\eta,Y_\ell^0\rangle_{L_2(\mathbb S^2)}
		=\langle\Theta D_a\eta,\Theta Y_\ell^0\rangle_{L_2(\R^2)}\\
		&=\langle D_a^+\Theta\eta,\Theta Y_\ell^0\rangle_{L_2(\R^2)}
		=\int_0^{2\pi}\int_0^\infty(\Theta Y_\ell^0)(r,\varphi)
		\overline{(\Theta\eta)(r/a,\varphi)}\,\frac{dr}{r}\,d\varphi.
	\end{align*}
	As $Y_\ell^0$ does not depend on $\varphi$, we may assume w.l.o.g. that $\eta$ is axisymmetric,
	otherwise we simply replace $\eta$ by $\widetilde\eta$ as
        defined by~\eqref{def:eta-tilde}.  Then the $\varphi$ integration just
	produces a constant factor $2\pi$. 
	Moreover, denoting
	$\check g(r):=\overline{g(r^{-1})}$, 
	we obtain, up to a normalizing factor depending on $\ell$,
	$$\widehat{\eta_a}(\ell,0)
	\sim 2\pi\int_0^\infty\bigl(\Theta[P_\ell(\cos\cdot)]\bigr)(r)
	(\Theta\eta)^\vee(a/r)\,\frac{dr}{r}.$$
	The last integral can be understood as the convolution of the functions
	$\Theta[P_\ell(\cos\cdot)]$ and $(\Theta\eta)^\vee$, defined on the multiplicative group
	$\R_*^+$ equipped with the corresponding Haar measure $\frac{dr}{r}$.
		Thus, for $G_\ell=0$ to be true for some $\ell\in\N_0$, this convolution needs to vanish for every
	$a>0$. However, the convolution of two functions can only vanish if their respective
	Mellin transforms have disjoint support. 
	
	\textbf{Step 2:} The Mellin transform of a function $g:(0,\infty)\rightarrow\C$ which is locally integrable
	w.r.t. the Haar measure	$\frac{dr}{r}$, is defined by
	$$(\mathcal{M}f)(s)=\int_0^\infty r^s g(r)\,\frac{dr}{r}.$$
	It is then easy to check that $\mathcal{M}(f\ast g)(s)=\mathcal{M}f(s)\mathcal{M}g(s)$, where
	$f\ast g$ is the convolution in $L_1(\R_*^+,\tfrac{dr}{r})$. We are now interested in the Mellin
	transform of
	$$\Theta[P_\ell(\cos\cdot)](r)=\frac{2r}{1+r^2}P_\ell\Bigl(\frac{1-r^2}{1+r^2}\Bigr).$$
	Using the substitution $\rho=\frac{1-r^2}{1+r^2}$, i.e., $r^2=\frac{1-\rho}{1+\rho}$, we obtain
	\begin{align*}
		\mathcal{M}\bigl[\Theta[P_\ell(\cos\cdot)]\bigr](s)
		&=\int_0^\infty r^s\frac{2r}{1+r^2}P_\ell\Bigl(\frac{1-r^2}{1+r^2}\Bigr)\,\frac{dr}{r}\\
		&=\int_{-1}^1\Bigl(\frac{1-\rho}{1+\rho}\Bigr)^{s/2}(1+\rho)
		P_\ell(\rho)\frac{1}{(1-\rho)^{1/2}(1+\rho)^{3/2}}d\rho\\
		&=\int_{-1}^1(1-\rho)^{\frac{s-1}{2}}(1+\rho)^{-\frac{s+1}{2}}P_\ell(\rho)\,d\rho\,.
	\end{align*}
	From this we can conclude that the integral converges precisely for $-1<\Re s<1$, since the
	Legendre-Polynomials are non-vanishing in the endpoints. Moreover, it is readily checked that it
	actually defines an analytic function on this open strip. As a consequence, the set of roots of
	$\mathcal{M}[\Theta[P_\ell(\cos\cdot)]]$ cannot contain a cluster point.
	
	For functions $f\in L_1((0,\infty),\frac{dr}{r})\cap L_2((0,\infty),\frac{dr}{r})$, the
	particular line $s=\frac{1}{2}+it$,	$t\in\R$, is always contained in	the domain of convergence.
	This gives rise to an isometry $\widetilde{\mathcal{M}}:L_2((0,\infty),\tfrac{dr}{r})\rightarrow L_2(\R)$ with
	$$\widetilde{\mathcal{M}}f(t)
	:=\frac{1}{\sqrt{2\pi}}\int_0^\infty r^{\frac{1}{2}+it}f(r)\,\frac{dr}{r}\,.
	$$
	This version of the Mellin transform inherits the convolution property, i.e.,
	$$\sqrt{2\pi}\widetilde{\mathcal{M}}(f\ast g)(t)
	=\widetilde{\mathcal{M}}f(t)\widetilde{\mathcal{M}}g(t),$$
	for almost all $t\in\R$ and for all $f,g\in L_2((0,\infty),\frac{dr}{r})$.
	
	Applying this general theory to our present situation we conclude that the function
	$\widetilde{\mathcal{M}}\bigl[\Theta[P_\ell(\cos\cdot)]\bigr]$ has global support.
	Therefore
	$\widehat{\eta_a}(\ell,0)$ vanishes for all $a>0$ 
	exactly when
	$\widetilde{\mathcal{M}}[(\Theta\eta)^\vee]=0$ and hence only for $\eta\equiv 0$.
\end{proof}

For convenience, we provide an alternative proof of the proposition in the appendix.

To prove that $(G_\ell)_\ell$ 
is bounded below by a strictly positive constant 
it suffices by Proposition \ref{prop-strictly-positive} to show that $\liminf_{\ell\to\infty}G_\ell>0$.
This is the content of our final theorem.

\begin{thm}\label{thm:lower}
	Let $\eta\in L_2(\mathbb S^2)$ \red{satisfy~\eqref{eq:7}.}
	Then 
	\begin{align}
		\liminf_{\ell\to\infty}\frac{1}{2\ell+1}
		&\int_0^\infty |\langle Y_\ell^0, \eta_a \rangle|^2 \, \frac{da}{a^3}\nonumber\\
		&\geq 4\pi\int_0^\infty
		\left( \int_0^\infty \frac{2t}{1+t^2}\, J_0(2ct) \widetilde\eta(2\arctan t)\, dt \right)^2
		\, \frac{dc}{c}\,,\label{eq:31}
	\end{align}
\red{where $\widetilde\eta$ is defined by~\eqref{def:eta-tilde}.       
Furthermore, if~\eqref{eq:11} holds true,}  the sequence $(G_\ell)_{\ell\in\N_0}$
	is bounded below by a strictly positive number.
\end{thm}

\begin{proof}
	\textbf{Step 1:} 
	For deriving a lower bound for the numbers $G_\ell$, it suffices 
	to consider in \eqref{fourier} the term
	with $m=0$, more precisely,
	$$G_\ell\geq\frac{1}{2\ell+1}\int_0^1 |\langle Y_\ell^0, \eta_a \rangle |^2 \frac{da}{a^3}\,.$$
	For the scalar product we obtain
	\[
	\langle Y_\ell^0, \eta_a \rangle
	=\langle D_{a^{-1}}Y_\ell^0, \eta\rangle
	=\int_{\mathbb S^2} \kappa(a^{-1},\theta) \sqrt{\frac{(2\ell+1)}{4\pi}}
	P_\ell(\cos(\theta_a)) \eta(\theta,\varphi)\, d\Sigma\,.
	\]
	To simplify the notation, we assume $\eta$ to be axisymmetric. Then we get
	\begin{align*}
		G_\ell
		&\geq \frac{1}{2\ell+1} \int_0^\infty
		|\langle Y_\ell^0, \eta_a \rangle|^2 \, \frac{da}{a^3}\\
		&= \frac1{4\pi}\int_0^\infty \left[
		\int_{\mathbb S^2} \kappa (a^{-1},\theta)^{\frac12} P_\ell(\cos(\theta_a)) \eta(\theta)
		\, d\Sigma \right]^2 \frac{da}{a^3}\\
		&=\frac1{4\pi}\int_0^\infty \left[
		\int_0^{2\pi} \, \int_0^\pi
		a^{-1} \kappa(a^{-1},\theta)^{\frac12}\, P_\ell(\cos\theta_a)\, \eta(\theta)\,
		\sin\theta\, d\theta\, d\varphi \right]^2 \frac{da}{a}\\
		&=\pi \int_0^\infty \left[	\int_0^\pi
		a^{-1} \kappa(a^{-1},\theta)^{\frac12}\, P_\ell(\cos\theta_a)\, \eta(\theta)\,
		\sin\theta\, d\theta \right]^2 \frac{da}{a}\\
		&=\pi\int_0^\infty \left[ \int_0^\pi
		\frac{2a^{-2}}{(a^{-2}-1)\cos\theta+(a^{-2}+1)} P_\ell(\cos\theta_a)\, \eta(\theta)\,
		\sin\theta\, d\theta \right]^2 \frac{da}{a}\\
		&=4\pi \int_0^\infty \left[ \int_0^\pi
		\frac{\sin\theta}{(1-a^2)\cos\theta+(1+a^2)} P_\ell(\cos\theta_a)\, \eta(\theta)\,
		d\theta \right]^2 \frac{da}{a}\,.
	\end{align*}
	As a next step, we deal with the inner integral.
	
	\vspace{0.3\baselineskip}\noindent
	\textbf{Step 2:}	For every $\ell\in\N$, we define $f_\ell:\R^+\to\R^+$ as 
	\begin{equation}\label{eq:f_ell}
		f_\ell(a)
		=\left[	\int_{0}^{\pi} 
		\frac{\sin\theta}{(1-a^2)\cos\theta+(1+a^2)} P_\ell(\cos\theta_a)\, \eta(\theta)\,
		d\theta \right]^2.
	\end{equation}
	For brevity, we set
	\begin{equation*}
		\kappa_a(\theta)
		:=\frac{\sin\theta}{1+\cos\theta+a^2(1-\cos\theta)}.
	\end{equation*}
	Then 
	$$f_\ell(a)
	=\left[ \int_0^\pi \kappa_a(\theta) P_\ell(\cos\theta_a) \eta(\theta)\, d\theta \right]^2,$$
	and our aim becomes to show that $\int_0^1 \frac{f_\ell(a)}{a} \,da$ has a strictly positive
	lower bound which is independent of $\ell$. We have for all $\theta\in(0,\pi)$ that
	\begin{align}
		\lim_{a\to 0}\kappa_a(\theta)
		&=\kappa_0(\theta)\label{eq:k-lim},\\
		0\leq \kappa_a(\theta)
		&\leq \kappa_0(\theta)=\tan\frac\theta2.\label{eq:k-bound}
	\end{align}
	Using the Taylor expansion of $\theta_a = 2\arctan\left(a\, \tan\frac{\theta}{2}\right)$ with
	respect to $a$ at $a=0$, we see that
	\begin{equation}
		\theta_a = 2a\, \tan\frac\theta2 + \mathcal{O}_\theta(a^3),
		\label{eq:xi-approx}
	\end{equation}
	where the subscript of $\mathcal{O}_\theta$ says that the constant of the Landau symbol $\mathcal{O}$
	depends on $\theta$. The formula \cite[Theorem 8.21.6]{Sz} states that
	\begin{equation*}
		P_\ell(\cos\zeta)
		= \sqrt{\frac{\zeta}{\sin\zeta}}\, J_0\left(\left(\ell+\frac12\right)\zeta\right)
		+ \mathcal O(\ell^{-\frac32}),
	\end{equation*}
	uniformly in $\zeta\in[0, \pi-\varepsilon]$ for some fixed $\varepsilon>0$, where $J_0$ denotes
	the Bessel function of order $0$.
	
	Let $c>0$. Then we conclude for $a=\frac c\ell$ that
	\begin{align*}
		P_\ell(\cos\theta_{\frac c\ell})
		&=\sqrt{\frac{\theta_{\frac c\ell}}{\sin\theta_{\frac c\ell}}}\,
		J_0\left(\left(\ell+\frac12\right)\theta_{\frac c\ell}\right)
		+ \mathcal{O}\left(\ell^{-\frac32}\right)\\
		&= \sqrt{\frac{\theta_{\frac c\ell}}{\sin\theta_{\frac c\ell}}}\,
		J_0\left(\frac{(2\ell+1)c}{\ell} \tan\frac{\theta}{2}
		+ \mathcal{O}_\theta\left(\frac {c^3}{\ell^2} \right)\right)
		+ \mathcal{O}\left(\ell^{-\frac32}\right),
	\end{align*}
	where the last line follows from \eqref{eq:xi-approx}. Taking the limit $\ell\to\infty$, we
	obtain for all fixed $\theta\in[0,\pi)$ in view of
	$\lim_{\ell\to\infty} \theta_{\frac c\ell}=0$ that
	\begin{equation}
		\lim_{\ell\to\infty} P_\ell(\cos \theta_{\frac c\ell})
		= J_0\biggl(2c \tan\frac{\theta}{2}\biggr).
		\label{eq:P-lim}
	\end{equation}
	Combining \eqref{eq:k-lim} and \eqref{eq:P-lim}, we get the pointwise limit
	\begin{equation*}
		\lim_{\ell\to\infty}
		\kappa_{\frac c\ell}(\theta)\, P_\ell(\cos\theta_{\frac c\ell})\, \eta(\theta)
		= \kappa_0(\theta)\, J_0\biggl(2c \tan\frac\theta2\biggr) \eta(\theta).
	\end{equation*}
	Next, the bound $|P_\ell(x)|\leq 1$ for all $x\in[-1,1]$ together with \eqref{eq:k-bound}
	implies
	\begin{equation}
		|\kappa_{\frac c\ell}(\theta)\, P_\ell(\cos\theta_{\frac c\ell})\, \eta(\theta)|
		\le |\kappa_{0}(\theta)\, \eta(\theta)|.
		\label{eq:f-ell-bound}
	\end{equation}
	The function $|\kappa_0\, \eta|$ is integrable on $(0,\pi)$ by assumption, and hence, by
	Lebesgue's Dominated Convergence Theorem, we obtain
	\begin{align*}
		\lim_{\ell\to\infty} f_\ell\left(\frac{c}{\ell}\right)
		&= \left( \int_0^\pi \lim_{\ell\to\infty} \kappa_{\frac c\ell}(\theta)\,
		P_\ell(\cos\theta_{\frac1\ell})\, \eta(\theta)\, d\theta \right)^2\\
		&= \left( \int_0^\pi \kappa_0(\theta)\, J_0\biggl(2c\tan\frac\theta2\biggr) \eta(\theta)\,
		d\theta\right)^2 =:f(c).
	\end{align*}
	
	\noindent
	\textbf{Step 3:} We want to estimate the integral
	$\int_0^\infty\frac{f_\ell(a)}{a}\,da$. For every fixed $k>0$ and $b>1$, we have
	\begin{equation*}
		\int_0^\infty f_\ell(a)\,\frac{da}{a}
		=\int_0^\infty f_\ell\Bigl(\frac{c}{\ell}\Bigr)\, \frac{dc}{c}
		\geq\sum_{j=-k}^k\int_{b^{j-1}}^{b^j}f_\ell\Bigl(\frac{c}{\ell}\Bigr)\, \frac{dc}{c}
		\geq\sum_{j=-k}^k b^{-j}\int_{b^{j-1}}^{b^j} f_\ell\Bigl(\frac{c}{\ell}\Bigr)\, dc\,.
	\end{equation*}
	Since $f_\ell(\frac{c}{\ell})$ converges to $f(c)$ for $\ell\to\infty$ pointwise
	with respect to $c$, and using \eqref{eq:f-ell-bound} to estimate
	$$\left|f_\ell\left(\frac{c}{\ell}\right)\right|
	\leq\left(\int_0^\pi|\kappa_0(\theta)\,\eta(\theta)|\,d\theta\right)^2,$$
	independently of $\ell$, we verify for every $j$ that
	\begin{equation*}
		\lim_{\ell\to\infty} \int_{b^{j-1}}^{b^j} f_\ell\left(\frac{c}{\ell }\right) \,dc
		= \int_{b^{j-1}}^{b^j} f(c)\, dc
		= \int_{b^{j-1}}^{b^j} \left( \int_0^\pi
		\kappa_0(\theta)\, J_0\biggl(2c\tan\frac\theta2\biggr) \eta(\theta)\,
		d\theta \right)^2 \,dc .
	\end{equation*}
	This implies 
	\begin{align}
		\liminf_{\ell\to\infty} \int_0^\infty \frac{f_\ell(a)}{a}\, da
		&\ge \sum_{j=-k}^k\int_{b^{j-1}}^{b^j}b^{-j}
		\left( \int_0^\pi \frac{\sin\theta}{1+\cos\theta}\, J_0\biggl(2c\tan\frac\theta2\biggr)
		\eta(\theta)\, d\theta \right)^2\, dc \nonumber\\
		&\geq \frac{1}{b}\int_{b^{-k-1}}^{b^k}
		\left( \int_0^\infty \frac{2t}{1+t^2}\, J_0(2ct) \eta(2\arctan t)\, dt \right)^2
		\, \frac{dc}{c}\,,
		\label{eq:int-J0}
	\end{align}
	where the last line uses the substitution $t=\tan\frac\theta2$ leading to
	$\cos\theta=\frac{1-t^2}{1+t^2}$, $\sin\theta=\frac{2t}{1+t^2}$, and
	$d\theta = \frac{2\,dt}{1+t^2}$. Noting that \eqref{eq:int-J0} is valid for all $k$, 
	we 	conclude
	\begin{equation}
		\liminf_{\ell\to\infty} \int_0^\infty \frac{f_\ell(a)}{a}\, da
		\geq \frac{1}{b}\int_0^\infty
		\left( \int_0^\infty \frac{2t}{1+t^2}\, J_0(2ct) \eta(2\arctan t)\, dt \right)^2
		\, \frac{dc}{c}\,.
		\label{eq:int-J0-2}
	\end{equation}
	Further, the fact that this estimate holds for all $b>1$
        implies \red{the lower bound~\eqref{eq:31}.}
	
	\vspace{0.3\baselineskip}\noindent
\textbf{Step 4:}	The inner integral in \eqref{eq:int-J0-2} is the Hankel transform of the function 
	$g(t):= \frac{2}{1+t^2} \eta(2\arctan t),$ evaluated at $2c$.
	Parseval's theorem for the Hankel transform \cite{MacAuley1939} states
	$$
	\int_0^\infty t |g(t)|^2\,  d t
	= \int_0^\infty c \, \left|\int_0^\infty  t g(t) J_0(ct) \,d t \right|^2  \, d c.
	$$
	This implies that for any non-zero function $\eta$ (which particularly entails $g\not\equiv 0$)
	the right-hand-side integral is not zero. Thus, there exists $1<B<\infty$ such that 
	\[
	\int_0^B \, c\left|\int_0^\infty t g(t) J_0(2ct)\, dt \right|^2 dc>0\,,
	\]
	which in turn yields
	\begin{align*}
		\int_0^B \, \left|\int_0^\infty t g(t) J_0(2ct)\, dt \right|^2 \frac{dc}{c}
		\geq\frac{1}{B^2}\int_0^B \, c\left|\int_0^\infty t g(t) J_0(2ct)\, dt \right|^2 dc>0\,.
	\end{align*}
	Ultimately we conclude that \eqref{eq:int-J0-2} does not vanish so that
	\begin{equation*}\liminf_{\ell\to\infty} G_\ell
		\ge 4\pi\liminf_{\ell\to\infty} \int_0^\infty f_\ell(a)\,\frac{da}{a} >0\,.
		\tag*{\qedhere}
	\end{equation*}
\end{proof}

\begin{rem}
	We conjecture that (under slightly stronger assumptions on $\eta$) the limit 
	%$\frac{1}{2\ell+1} \int_0^\infty |\langle Y_\ell^0, \eta_a \rangle|^2 \, \frac{da}{a^3}$ as $\ell \rightarrow \infty$ 
	indicated below exists,
	and
	$$\lim_{\ell\to\infty}\frac{1}{2\ell+1}
	\int_0^\infty |\langle Y_\ell^0, \eta_a \rangle|^2 \, \frac{da}{a^3}
	=4\pi\int_0^\infty \left(
	\int_0^\pi \kappa_0(\theta)\, J_0\biggl(2c\tan\frac\theta2\biggr)\,
	\widetilde\eta(\theta)\, d\theta\right)^2\,\frac{dc}{c}\,.$$
	A proof of such a	result would again require to determine an upper bound  for
	$f_\ell(\frac{c}{\ell})$, but integrable w.r.t. $\frac{dc}{c}$, which in turn would lead to an
	alternative (and stronger) proof for the upper frame bound.
\end{rem}

Let us summarize the results leading to the proof of our main theorem.
\\

\emph{Proof of Theorem} \ref{thm:main}.
Assertion i) of our main theorem is a consequence of Proposition \ref{lemma-asymp-1}
and Theorem \ref{thm:upper}.

Assertion ii) of Theorem \ref{thm:main} was shown in one direction in Theorem \ref{thm:lower}.
\red{To prove the ``only if''-part  assume that
\[
 \int_0^{2\pi}\eta(\theta,\varphi)\,d\varphi=0
\]
for almost all $\theta\in [0,\pi]$. Then
\[
\langle\eta_a,Y_0^0\rangle=
\langle\eta,D_{1/a}Y_0^0\rangle=\langle\eta,P_0\rangle =
\int_0^\pi\bigl(\int_0^{2\pi}\eta(\theta,\varphi)\,d\varphi\bigr)\sin\theta\,d\theta=0
\]
since $Y_0^0=P_0$ is a constant, so that $D_{1/a}Y_0^0=Y_0^0$.  Hence, 
$\|\Pi_0\eta_a\|^2 =0$ so that $G_0=0$}. \hfill $\Box$
\\

\red{Finally, we prove Lemma \ref{exist:eta}, which corresponds to
  Proposition~3.7 in \cite{AV99}. Here we provide a shorter proof. }
\\

\emph{Proof of Lemma} \ref{exist:eta}.
	The result is an immediate consequence of properties of the stereographic projection in Lemma
	\ref{lemma-stereo}, namely
	\begin{align*}
		\int_{\mathbb S^2}\frac{\zeta(\theta,\varphi)}{1+\cos\theta}\,d\Sigma(\omega)
		&=\int_0^{2\pi}\int_0^\infty(\Theta\zeta)(r,\varphi)\,dr\,d\varphi
		=\frac1\alpha\int_0^{2\pi}\int_0^\infty(\Theta\zeta)(s/\alpha,\varphi)\,ds\,d\varphi\\
		&=\frac1\alpha\int_0^{2\pi}\int_0^\infty(D_\alpha^+\Theta\zeta)(s,\varphi)\,ds\,d\varphi\\
		&=\frac1\alpha\int_0^{2\pi}\int_0^\infty(\Theta D_\alpha\zeta)(s,\varphi)\,ds\,d\varphi
		=\frac1\alpha\int_{\mathbb S^2}\frac{D_\alpha\zeta(\theta,\varphi)}{1+\cos\theta}\,d\Sigma(\omega)\,.
		\qquad \Box
	\end{align*}

%-------------------------------------------------------------------------------
\appendix
\section{Alternative proof of Proposition \ref{prop-strictly-positive}}
In addition to the transform $\Theta$ from Lemma \ref{lemma-stereo}, we define a second transform
	${\cal J}: L_2(\R_+\times[0,2\pi),r^{-1}dr\,d\varphi) \to L_2(\R\times[0,2\pi),dr\,d\varphi) $
	via
	\[
	({\cal J} f)(r,\varphi) = f(e^r , \varphi), \quad f\in L_2(\R_+\times[0,2\pi))\,.
	\]
	Then it is straightforward to check for arbitrary $b\in\R$ that
	$$ {\cal J} D_{e^b}^+ = T_b {\cal J},$$
	where $T_bf(r,\varphi) := f(r-b,\varphi)$.
	Setting $\widetilde{Y}_\ell^n := {\cal J}\Theta Y_\ell^n$, we find for every
	$(r,\varphi)\in \R\times [0,2\pi)$ that
	\begin{align}\label{def-phi}
		\widetilde{Y}_\ell^n (r,\varphi )
		&= \sqrt{ \frac{2\ell +1}{4\pi} \, \frac{(\ell-m)!}{(\ell + n)!} }
		\frac{(-1)^{\ell n}}{\cosh r}P_{\ell}^n(\tanh r) e^{in\varphi} \nonumber\\
		&= \sqrt{(2\ell + 1) \frac{(\ell-m)!}{(\ell + n)!} }
		(-1)^{\ell n} \phi_{\ell}^n (r) e^{in\varphi}\,,
	\end{align}
	where $\phi_{\ell}^n (r) := \frac{1}{\sqrt{{4\pi }}\, \cosh r}  P_{\ell}^n(\tanh r)$. 
		For simplicity let us once again assume that $\eta (\theta, \varphi) = \eta(\theta)$.
		In this case, we have
	\begin{align*}
		G_\ell
		&= \frac{1}{2\ell + 1} \int_{0}^\infty \| \Pi_\ell D_a \eta \|^2 \, \frac{da}{a^3}\\
		&= \frac{1}{2\ell + 1} \int_{\R}
		\left| \langle D_{e^b} \eta , Y_\ell^0 \rangle \right|^2 e^{-2b} \, db \\
		&= \frac{1}{2\ell + 1} \int_{\R}
		\left| \langle {\cal J}\Theta D_{e^b} \eta , {\cal J}\Theta Y_\ell^0 \rangle \right|^2
		e^{-2b} \, db,
	\end{align*}
	where the last equality holds since both $\Theta$ and $\cal J$ are isometries. 
	Let
	$\psi:= {\cal J}\Theta\eta$.
	Then for every $b\in\R$ it holds 
	$${\cal J} \Theta D_{e^b} \eta
	= {\cal J} D_{e^b}^+ \Theta \eta = T_b {\cal J} \Theta \eta = T_b \psi.$$
	By \eqref{def-phi}, we also have
	$${\cal J}\Theta Y_\ell^0 (r, \varphi)
	= \widetilde{Y}_\ell^0 (r,\varphi)=\sqrt{(2\ell + 1) } \phi_{\ell}^0 (r) .$$
	Therefore we obtain
	\begin{align*}
		G_\ell
		&= \int_{\R} \left| \langle T_b \psi , \phi_{\ell}^0 \rangle \right|^2 \, e^{-2b} \, db 
		= \int_{\R} \left| \langle {\cal F}T_b \psi ,
		{\cal F}\phi_{\ell}^0 \rangle \right|^2 \, e^{-2b} \, db \\
		&= \int_{\R} \left| \langle  M_{-b}{\cal F} \psi ,
		{\cal F}\phi_{\ell}^0 \rangle \right|^2 e^{-2b} \, db 
		= \int_{\R} \left| {\cal F}^{-1}
		\left( {\cal F}\phi_{\ell}^0 \,\overline{{\cal F} \psi}\right)(b)\right|^2 e^{-2b} \,db\,.
	\end{align*}
	Hence, for $G_\ell=0$, we need that the function
	${\cal F}^{-1}\left( {\cal F}\phi_{\ell}^0 \,\overline{{\cal F} \psi}\right)$, and thus
	${\cal F}\phi_{\ell}^0 \,\overline{{\cal F} \psi}$, vanishes almost everywhere. However, since
	$\phi_\ell^0(r)\leq e^{-r}$, the first factor ${\cal F}\phi_{\ell}^0$ is an analytic function,
	more precisely, it can be extended to an analytic function on the strip
	$\{z\in\C:|\Re z|<1\}$. 
	In particular, it has only isolated roots. This implies
	that $G_\ell=0$ is only possible if $\cf\psi$ vanishes almost everywhere, which in turn
	contradicts the assumption $\widetilde \eta\not\equiv 0$.
\hfill $\Box$

\subsection*{Acknowledgments} 
G.S. gratefully acknowledges funding by the German Research Foundation (DFG) with\-in the project STE 571/16-1
and within Germany's Excellence Strategy – The Berlin Mathematics Research Center MATH+ 
(EXC-2046/1, project ID: 390685689).

F.D.M. and E.D.V. are part of the Machine Learning Genoa Center (MaLGa) and
are  members of the Gruppo Nazionale per l'Analisi Matematica, la
Probabilit\`a e le loro Applicazioni (GNAMPA) of the Istituto Nazionale
di Alta Matematica (INdAM). 

M.H. gratefully acknowledges funding by the German Research Foundation (DFG) with\-in the project Da 360/22-1.

\bibliographystyle{abbrv}
\bibliography{references}

\begin{thebibliography}{10}

\bibitem{AAG2000}
S.~T. Ali, J.-P. Antoine, and J.-P. Gazeau.
\newblock {\em Coherent States, Wavelets and Their Generalizations}.
\newblock Springer, Providence, RI, USA, 2000.

\bibitem{antoine1998wavelets}
J.-P. Antoine and P.~Vandergheynst.
\newblock Wavelets on the n-sphere and related manifolds.
\newblock {\em J. Math. Phys.}, 39(8):3987--4008, 1998.

\bibitem{AV99}
J.-P. Antoine and P.~Vandergheynst.
\newblock Wavelets on the $2$-sphere: A group-theoretical approach.
\newblock {\em Appl. Comp. Harm. Anal.}, 7:262--291, 1999.

\bibitem{Bernstein2009}
S.~Bernstein.
\newblock Spherical singular integrals, monogenic kernels and wavelets on the
  three–dimensional sphere.
\newblock {\em Adv. Appl. Clifford Algebr.}, 19(2):173--189, 2009.

\bibitem{BE2010}
S.~Bernstein and S.~Ebert.
\newblock Wavelets on {$S^3$} and {$\SO(3)$} -- their construction, relation to
  each other and radon transform of wavelets on {$\SO(3)$}.
\newblock {\em Math. Methods Appl. Sci.}, 33(16):1895–--1909, 2010.

\bibitem{DDSW1995}
S.~Dahlke, W.~Dahmen, E.~Schmitt, and I.~Weinreich.
\newblock Multiresolution analysis and wavelets on {$S^2$} and {$S^3$}.
\newblock {\em Numer. Funct. Anal. Optim.}, 16:19--41, 1995.

\bibitem{DM1996}
S.~Dahlke and P.~Maass.
\newblock Continuous wavelet transforms with applications to analyzing
  functions on spheres.
\newblock {\em J. Fourier Anal. Appl.}, 2:379--396, 1996.

\bibitem{DST2004}
S.~Dahlke, G.~Steidl, and G.~Teschke.
\newblock Coorbit spaces and {B}anach frames on homogeneous spaces with
  applications to the sphere.
\newblock {\em Adv. Comput. Math.}, 21:147--180, 2004.

\bibitem{duflo1976regular}
M.~Duflo and C.~C. Moore.
\newblock On the regular representation of a nonunimodular locally compact
  group.
\newblock {\em Journal of functional analysis}, 21(2):209--243, 1976.

\bibitem{folland2016course}
G.~B. Folland.
\newblock {\em A course in abstract harmonic analysis}, volume~29.
\newblock CRC press, 2016.

\bibitem{FSG1997}
W.~Freeden, M.~Schreiner, and T.~Gervens.
\newblock {\em Constructive Approximation on the Sphere, with Applications to
  Geomathematics}.
\newblock Clarendon Press, 1997.

\bibitem{FW1997}
W.~Freeden and U.~Windheuser.
\newblock Combined spherical harmonic and wavelet expansion—a future concept
  in earth’s gravitational determination.
\newblock {\em Appl. Comput. Harmon. Anal.}, 4:1--37, 1997.

\bibitem{fuhr2005abstract}
H.~F{\"u}hr.
\newblock {\em Abstract harmonic analysis of continuous wavelet transforms},
  volume 1863.
\newblock Springer Science \& Business Media, 2005.

\bibitem{Go1999}
J.~G\"ottelmann.
\newblock Locally supported wavelets on manifolds, with applications to the
  {2D} sphere.
\newblock {\em Appl. Comput. Harmon. Anal.}, 7:1--33, 1999.

\bibitem{GMP1985}
A.~Grossmann, J.~Morlet, and T.~Paul.
\newblock Integral transforms associated to square integrable representations.
  {I. G}eneral results.
\newblock {\em J. Math. Phys.}, 26:2473--2479, 1985.

\bibitem{GMP1986}
A.~Grossmann, J.~Morlet, and T.~Paul.
\newblock Integral transforms associated to square integrable representations.
  {II. E}xamples.
\newblock {\em Ann. Inst. H. Poincar\'e}, 45:293--309, 1986.

\bibitem{holschneider1996continuous}
M.~Holschneider.
\newblock Continuous wavelet transforms on the sphere.
\newblock {\em J. Math. Phys.}, 37(8):4156--4165, 1996.

\bibitem{Nowak2015}
I.~Iglewska-Nowak.
\newblock Continuous wavelet transforms on n-dimensional spheres.
\newblock {\em Appl. Comput. Harmon. Anal.}, 39(2):248--276, 2015.

\bibitem{knapp2013lie}
A.~W. Knapp.
\newblock {\em Lie groups beyond an introduction}, volume 140.
\newblock Springer Science \& Business Media, 2013.

\bibitem{MacAuley1939}
P.~MacAulay-Owen.
\newblock Parseval's theorem for {H}ankel transforms.
\newblock {\em Proc. London Math. Soc.}, 45(1):458--474, 1939.

\bibitem{MNPW2000}
H.~N. Mhaskar, F.~J. Narcowich, J.~Prestin, and J.~D. Ward.
\newblock Polynomial frames on the sphere.
\newblock {\em Adv. Comput. Math.}, 13(4):387--403, 2000.

\bibitem{NW1996}
F.~J. Narcowich and J.~D. Ward.
\newblock Nonstationary wavelets on the m-sphere for scattered data.
\newblock {\em Appl. Comput. Harmon. Anal.}, 3:1324--1336, 1996.

\bibitem{PST1996}
D.~Potts, G.~Steidl, and M.~Tasche.
\newblock Kernels of spherical harmonics and spherical frames.
\newblock In {\em Advanced Topics in Multivariate Approximation}, pages 1--154.
  World Scientific, 1996.

\bibitem{PT1995}
D.~Potts and M.~Tasche.
\newblock Interpolatory wavelets on the sphere.
\newblock In {\em Approximation Theory VIII}, pages 335--342. World Scientific,
  1995.

\bibitem{SBHBKP2005}
H.~Schaeben, S.~Bernstein, R.~Hielscher, J.~Beckmann, J.~Keiner, and
  J.~Prestin.
\newblock High resolution texture analysis with spherical wavelets.
\newblock {\em Materials Sci. For.}, 495:245--254, 2005.

\bibitem{Sweldens1996}
W.~Sweldens.
\newblock The lifting scheme: A custom-design construction of biorthogonal
  wavelets.
\newblock {\em Appl. Comput. Harmon. Anal.}, 3:1186--1200, 1996.

\bibitem{Sz}
G.~Szeg{\H{o}}.
\newblock {\em Orthogonal Polynomials}.
\newblock Amer. Math. Soc., Providence, RI, USA, 4th edition, 1975.

\end{thebibliography}
\end{document}